\documentclass{amsart}[12 pt]

\usepackage{latexsym}
\usepackage{amsmath,amssymb,dsfont,amsfonts}
\usepackage{mathrsfs}
\usepackage{verbatim}
\usepackage{graphicx}
\usepackage{graphics}
\usepackage{geometry}
\usepackage{booktabs}
\usepackage{enumitem}
\usepackage{xcolor}

\newtheorem{proposition}{Proposition}[section]
\newtheorem{theorem}{Theorem}[section]
\newtheorem{definition}{Definition}[section]
\newtheorem{corollary}{Corollary}[section]
\newtheorem{lemma}{Lemma}[section]
\newtheorem{remark}{Remark}[section]
\newtheorem{example}{Example}[section]

\DeclareMathOperator*{\esssup}{ess\,sup} % copiar esta linha

\numberwithin{equation}{section}

\allowdisplaybreaks

\begin{document}

\markboth{}{}

\title[Discrete-type approximations for non-Markovian optimal stopping problems: Part I ] {Discrete-type approximations for non-Markovian optimal stopping problems: Part I}

\author{Dorival Le\~ao}

%%% ADAPT THE ADDRESS OF LEAO
\address{Estatcamp. Rua Maestro Jo\~ao Seppe, 900, 13561-180, S\~ao Carlos, Brazil.}\email{leao@estatcamp.com.br}

\author{Alberto Ohashi}

\address{Departamento de Matem\'atica, Universidade de Bras\'ilia, 70910-900, Bras\'ilia, Brazil.}\email{amfohashi@gmail.com}

\author{Francesco Russo}

\address{ENSTA ParisTech, Institut Polytechnique de Paris,
 Unit\'e de Math\'ematiques appliqu\'ees, 828, boulevard des Mar\'echaux,
F-91120 Palaiseau, France} \email{francesco.russo@ensta-paristech.fr}

%\thanks{Corresponding author: Alberto Ohashi}
\date{\today}


\begin{center}
\end{center}
\begin{abstract}
In this paper, we present a discrete-type approximation scheme to solve continuous-time optimal stopping problems based on fully non-Markovian continuous processes adapted to the Brownian motion filtration. The approximations satisfy suitable variational inequalities which allow us to construct $\epsilon$-optimal stopping times and optimal values in full generality. Explicit rates of convergence are presented for optimal values based on reward functionals of path-dependent SDEs driven by fractional Brownian motion. In particular, the methodology allows us to design concrete Monte-Carlo schemes for non-Markovian optimal stopping time problems as demonstrated in the companion paper by Bezerra, Ohashi and Russo.
\end{abstract}
\maketitle
\section{Introduction} % Initial capital letter, then lower case. No full stop.

% Write the text of your paper using normal LaTeX commands.
% For instance, you can use the `\cite' command~\cite{ref1}.
% When giving citations a numbering system is preferred~\cite{ref2},
% but an author--date system is also acceptable~\cite{ref3}.

An important and well-developed class of stochastic control problems is related to
 optimal stopping time. The goal is to find stopping
times, at which an underlying stochastic process $Z$ should be stopped in order to optimize the values of some given functional of interest. Essentially, the optimal stopping time problem is completely described by the so-called supermartingale Snell envelope

\begin{equation}\label{snelllint}
S (t):= \text{ess} \sup_{\tau\ge t} \mathbb{E} \left[ Z(\tau)  \mid \mathcal{F}_t \right], \quad 0 \leq t \leq T,
\end{equation}
written w.r.t a given filtration $\mathbb{F}=(\mathcal{F}_t)_{0\le t\le T}$, where esssup in (\ref{snelllint}) is computed over the class of all $\mathbb{F}$-stopping times located on $[t,T]$. The literature on probabilistic techniques in optimal stopping problems is vast. We
refer for instance to the expositions of \cite{lamberton, carmona} and other references therein for a broad overview of the techniques in this area.

One typical approach in solving an optimal stopping time problem is to discretize the Snell envelope (\ref{snelllint}) along a discrete-time period $t_0, \ldots, t_p$ by means of Euler-type discretization schemes. The main obstacle is the obtention of continuation values (expressed as $\mathcal{F}_{t_i}$- conditional expectations) in the dynamic programming algorithm. In this case, least-squares Monte-Carlo methods can be employed by using non-parametric regression techniques based on a suitable choice of regression polynomials (see e.g \cite{zanger} and other references therein). Another popular approach is to make use of some available representations for conditional expectations in terms of a suitable ratio of unconditional expectations obtained by using Malliavin calculus (see e.g \cite{bouchard2}). These approaches work well as long as one may reduce the information flow $(\mathcal{F}_{t_i})_{i=0}^p$ by a finite-dimensional random vector $X(t_0), \ldots, X(t_p)$. This is only possible if the state is Markovian. Otherwise, the computation of continuation values is a priori unfeasible from a computational point of view. A dual approach studied in \cite{rogers} has also been investigated by many authors (see e.g.~\cite{belomestny} and other references therein). In this approach, the key step is to find the ``optimal" martingale and the importance of the Markov property seems critical to obtain concrete basis functions to parameterize martingales (see e.g \cite{belomestny,schoenmakers}).

%\textcolor{red}{Applications of these methods to non-Markovian states are not implementable due to the notorious difficulty in handling conditional expectations which could possibly depend on the whole history of the states.}

%In that framework there are also contributions
%making use of discretizations as those of Lamberton-Pag\`es \cite{lamberton1} and Mulinacci-Pratelli \cite{mulinacci},

Beyond Markovian state processes, one  approach consists in expressing the Snell envelope $S$ as the unique solution of a reflected BSDE. In one hand, at a representation level, the theory of reflected BSDEs characterizes the semimartingale structure of $S$ in full generality. On the other hand, concrete methods to compute martingale representations are restricted to the Markovian case (see e.g \cite{bally,bally1,bouchard,bouchard1}). More recent approaches based on path-dependent PDEs (see e.g.~\cite{ekren}) yield differential representations of the Snell envelope (even in the fully nonlinear case) under pathwise regularity on the obstacles related to the reflected BSDE. We also emphasize the recent representation result by \cite{fuhrman} for the Snell envelope in terms of true BSDEs defined on an enlarged probability space, containing a jump part and involving sign
constraints. We observe the use of reflected BSDEs and path-dependent PDEs to solve concretely the optimal stopping problem beyond the Markovian case seems to be a very hard task. The key difficulty is that solving an optimal stopping time problem for a fully non-Markovian state is essentially a stochastic optimal control problem driven by an infinite-dimensional state.

%\textcolor{red}{Quantization methods (see e.g \cite{bally}) are restricted to Markov diffusions because in this case conditional expectations become tractable. Otherwise, one has to employ functional quantization methods which are not feasible from a computational point of view}.

Optimal stopping time problems based on reward path-dependent functionals of fully non-Markovian states arise in many applications and it remains a challenging task to design concrete ways of solving it. One typical example is the problem of pricing American options written on stochastic volatility models $(X,V)$ where the volatility variable $V$ is a functional of fractional Brownian motion (FBM) with long-range dependence (see e.g \cite{viens} and other references therein) or short-range dependence as recently observed by many works (see e.g. \cite{bayer,gatheral,forde} and other references therein). In this case, one cannot interpret $(X,V)$ as an ``augmented" finite-dimensional Markov process due to non-trivial correlations associated with FBM and hence a concrete solution of the related optimal stopping problem is very challenging. We stress even optimal stopping problems based on finite-dimensional ``augmented" Markovian-type models $(X,V)$ (e.g CIR-type models driven by Brownian motion) are not easy to solve it (see e.g \cite{rambharat,agarwal}). In practice, the volatility $V$ is not directly observed so that it has to be approximated. If the volatility is estimated with accuracy, then standard Markovian methods can be fairly implemented to solve the control problem. Otherwise, the optimal stopping problem is non-Markovian.

%The main obstacle in solving a fully non-Markovian optimal stopping (not reducible to a finite dimensional Markov process) is to keep track the whole state history $\{X(s); s\le t\}$ which causes a severe difficulty inas long as the underlying state $X$ is Markovian (see e.g ). starting from the theoretical regression function $\mathbb{R}^d \ni x\mapsto \mathbb{E}[\cdot|X(t_i)=x]$Any attempt based on classical time discretizations and Markov chains schemes break down due to the practical impossibility to implement standard regression methods, representation formulae based on Malliavin calculus and quantization methods.

The goal of this paper is to present a systematic method of constructing continuation regions and optimal values for optimal stopping problems based on arbitrary continuous reward processes $Z$. We are particularly interested in the non-Markovian case, i.e., $Z = F(X)$ where $F$ is a path-dependent functional of an underlying non-Markovian continuous process $X$ which is adapted to a filtration generated by a $d$-dimensional Brownian motion $B$. In our philosophy, the reward process $Z$ is viewed as a generic non-anticipative functional of $B$, i.e.
\begin{equation}\label{ZB}
B\mapsto Z(B)
\end{equation}
rather than $X$ and it may be possibly based on different kinds of states not reducible to vectors of Markov processes: solutions of path-dependent SDEs, processes involving fractional Brownian motions and other singular functionals of Brownian motions.

The methodology is based on Le\~ao and Ohashi \cite{LEAO_OHASHI2013} and Le\~ao, Ohashi and Simas~\cite{LEAO_OHASHI2017.1} who show a very strong continuity/stability of a large class of Wiener functionals (\ref{ZB}) w.r.t continuous-time random walk structures $\mathscr{D}$ driven by suitable waiting times $T^k_n; n,k\ge 1$ (see (\ref{hittingtimes1}) and (\ref{hittingtimes2})) which encode the evolution of the Brownian motion at small scales $\epsilon_k\downarrow 0$ as $k\rightarrow +\infty$. Here, $\epsilon_k = \varphi(k)$ for a strictly decreasing function with inverse $\xi$. The main idea is to reduce the dimension by discretizing the Brownian filtration at a given time $t$ by

\begin{equation}\label{discreteFILTRATION}
\mathcal{F}^k_t:=\Big\{\bigcup_{j\ge 0}D_j \cap \{T^k_j \le t < T^k_{j+1}\}; D_j \in \sigma(\mathcal{A}^k_j), j\ge 0\Big\},
\end{equation}
where $\sigma(\mathcal{A}^k_j); j\ge 1$ is a sequence of sigma-algebras generated by random vectors

\begin{equation}\label{Aknint}
\mathcal{A}^k_j:= \Big(\Delta T^k_1, \eta^k_1, \ldots, \Delta T^k_j, \eta^k_j\Big)
\end{equation}
taking values on a cartesian product of a finite-dimensional discrete space (involving suitable signs $(\eta^k_j)_{j\ge 1}$ of $\mathscr{D}$) times $[0,T]$. In contrast to previous discretization schemes, our method has to be interpreted as a \textit{space-filtration} discretization scheme: (i) at random intervals $T^k_n \le t < T^k_n$, the Brownian motion lies on open subsets with dyadic end points of the form $\big((\ell-1) \epsilon_k , (\ell+1)\epsilon_k\big)$ for $\ell\in \mathbb{Z}$ and $-n \le \ell \le n$, (ii) $\mathcal{F}^k_t$ has the important property that $\mathcal{F}^k_t \cap \{T^k_n \le t < T^k_{n+1}\} = \sigma(\mathcal{A}^k_n) \cap \{T^k_n \le t < T^k_{n+1}\}; n\ge 0$.

The structure $\mathscr{D}$ is constructed from $\{\mathcal{A}^k_n; k,n\ge 1\}$ and the first step is to evaluate the optimal stopping time problem imbedded into the structure $\mathscr{D}$ where all functionals of the Brownian motion are replaced by functionals of $\{\mathcal{A}^k_n; n,k\ge 1\}$. Two classes of admissible stopping times are considered. The first class consists of all Brownian stopping times that take values in $[0,T]$. The second class further restricts the set of allowed values to the discrete grid $\{n; n=0,1, \ldots, e(k,T)\}$ where $e(k,T) = \lceil d\epsilon^{-2}_k T\rceil$ is the relevant number of steps associated with a suitable discrete structure $\big((S^k)_{k\ge 1},\mathscr{D}\big)$ for (\ref{snelllint}). The number of steps is determined by Large Deviation principles (see Lemma 2.2 in \cite{koshnevisan}, Lemma \ref{ratetenk} and (\ref{tail6})) and it only depends on the dimension $d$ of the Brownian motion and the discretization level $\epsilon_k$. The structure $\big((S^k)_{k\ge 1},\mathscr{D}\big)$ is typically constructed by means of Euler-type schemes driven by $\mathscr{D}$. The optimal values at time $t=0$ for the two problems are denoted by $S(0)$ and $V^k_0$, respectively.

The advantage of our methodology in comparison to classical approaches lies on a concrete analysis of optimal stopping problems for non-Markovian states: First, recalling our philosophy of interpreting states (\ref{ZB}) as Wiener functionals, the filtration $(\mathcal{F}^k_t)_{0\le t\le T}$ summarizes the whole Brownian history in such way that the (infinite-dimensional) information flow of the optimal stopping problem (\ref{snelllint}) can be reduced to the computation of conditional expectations w.r.t random vectors (\ref{Aknint}) along the time grid $n=0, \ldots, e(k,T)$. The non-Markovian states are then viewed as functionals of (\ref{Aknint}) and, in contrast to previous methods, continuation values associated with our approximating optimal stopping problem can be fairly computed by means of standard regression techniques based on a perfect simulation algorithm (Burq and Jones \cite{Burq_Jones2008}) for the first time Brownian motion hits $\pm 1$. This allows us to drastically reduce the dimension of the continuous-time non-Markovian problem in terms of a high-dimensional discrete-time dynamic programming algorithm. Second, the strong stability of Wiener functionals (\ref{ZB}) w.r.t $\mathscr{D}$ allows us to concretely treat non-Markovian states that are not within reach of the classical methods based on reflected BSDEs, path-dependent PDEs and other discretization methods.

Theorems \ref{mainTh} and \ref{optSTres} present the abstract results for a generic optimal stopping time problem with generic states admitting an imbedded discrete structure in the language of \cite{LEAO_OHASHI2017.1}. As a test of relevance of our theory, Section \ref{ExamplesSECTION} presents concrete examples based on path-dependent SDEs driven by fractional Brownian motion with parameter $\frac{1}{2}\le H < 1$. Propositions \ref{exBM} and \ref{exFBM} provide rates of convergence of $V^k_0$ to $S(0)$. More importantly, as explained in Section \ref{appliedsec}, $V^k_0$ can be concretely computed by means of the classical discrete-time dynamic programming principle over $0\le n\le e(k,T)$ based on the information set generated by (\ref{Aknint}).

The amount of work (complexity) of the algorithm in a typical non-Markovian state driven by a fractional Brownian motion, for a given accuracy $\text{e}_1$, it is proportional to $O(\epsilon^{1-2\lambda}_{k^*})$ for $H-\frac{1}{2}< \lambda  < \frac{1}{2}$ and $k^* =\xi(\text{e}^{\frac{1}{1-2\lambda}}_1)$ which in turn implies $e(k^*,T) = \lceil d \epsilon^{-2}_{k^*} T\rceil$ number of steps in the algorithm. The underlying state $\mathcal{A}^{k^*}_{e(k^*,T)}$ lives in a $e(k^*,T)(d +1)$-dimensional space. A concrete Monte-Carlo scheme is developed in the companion paper by Bezerra, Ohashi and Russo \cite{ohashi}. Lastly, we mention that there is no conceptual obstruction to compute sensitivities of $S$ w.r.t $B$ (hedging strategies) by projecting the optimization problem onto the filtration $(\mathcal{F}^k_t)_{0\le t\le T}$ and working with the differential operators introduced in \cite{LEAO_OHASHI2017.1} which describe Doob-Meyer decompositions for test processes. This will be investigated in a future project.

%\textcolor{red}{To the best of our knowledge, the results of this paper are the first to directly apply for constructing optimal values in optimal stopping time problems with non-Markovian states.

%In contrast to previous time discretization schemes for reflected BSDEs, quantization methods (see e.g \cite{bally,bouchard,bouchard1,bouchard2}) applied to Markov diffusions, our methodology provides a consistent and implementable discretization procedure for generic states depending on the Brownian motion.

%The methodology may also led to new perspectives on the dual approach of optimal stopping times and their
This paper is organized as follows. In Section \ref{preliminaries}, we present the basic objects of our discretization scheme. In Section \ref{varsection}, we present the main results of the paper and we explain how the methodology can be concretely used to solve optimal stopping problems beyond the Markovian case. In Section \ref{sectionproof}, we present the proofs of Proposition \ref{equality} and Theorem \ref{mainTh}. Section \ref{ExamplesSECTION} illustrates the methodology with two examples. Section \ref{appliedsec} explain how to operationalize the discrete-time dynamic programming principle in a concrete example.

\section{Preliminaries}\label{preliminaries}
Throughout this article, we are going to fix a $d$-dimensional Brownian motion $B = \{B^{1},\ldots,B^{d}\}$ on the usual stochastic basis $(\Omega, \mathbb{F}, \mathbb{P})$, where $\Omega$ is the space $C([0,+\infty);\mathbb{R}^d) := \{f:[0,+\infty) \rightarrow
\mathbb{R}^d~\text{continuous}\}$, equipped with the usual topology
of uniform convergence on compact intervals.
$\mathbb{P}$ is the Wiener measure
on $\Omega$ such that $\mathbb{P}\{B(0) = 0\}=1 $ and $\mathbb{F}:=(\mathcal{F}_t)_{t\ge 0}$ is the
usual $\mathbb{P}$-augmentation of the natural filtration
generated by the Brownian motion. In the sequel, we recall the basic structure of the theory presented by the works~\cite{LEAO_OHASHI2013, LEAO_OHASHI2017.1} and we refer the reader to these works for all unexplained points.

For a fixed positive sequence $\{\epsilon_k;k\ge 1\}$ such that $\sum_{k\ge 1}\epsilon^2_k< +\infty$, and for each $j = 1, \ldots, d$, we define $T^{k,j}_0 := 0$ a.s. and we set

\begin{equation}\label{hittingtimes1}
T^{k,j}_n := \inf\big\{T^{k,j}_{n-1}< t <\infty;  |B^{j}(t) - B^{j}(T^{k,j}_{n-1})| = \epsilon_k\big\}, \quad n \ge 1.
\end{equation}
For each $j\in \{1,\ldots,d \}$, the family $(T^{k,j}_n)_{n\ge 0}$ is a sequence of $\mathbb{F}$-stopping times and the strong Markov property implies the increments $\{T^{k,j}_n - T^{k,j}_{n-1}; n\ge 1\}$ is an i.i.d sequence with the same distribution as $T^{k,j}_1$. Moreover, $T^{k,j}_1$ is an absolutely continuous variable (see \cite{LEAO_OHASHI2017.1, borodin}).

From this family of stopping times, we define $A^{k}:=(A^{k,1},\ldots, A^{k,d})$ as the $d$-dimensional step process whose components are given by

$$
A^{k,j} (t) := \sum_{n=1}^{\infty}\epsilon_k~\sigma^{k,j}_n1\!\!1_{\{T^{k,j}_n\leq t \}};~t\ge0,
$$
where

$$
\sigma^{k,j}_n:=\left\{
\begin{array}{rl}
1; & \hbox{if} \ B^{j} (T^{k,j}_n) - B^{j} (T^{k,j}_{n-1}) > 0 \\
-1;& \hbox{if} \ B^{j} (T^{k,j}_n) - B^{j} (T^{k,j}_{n-1})< 0, \\
\end{array}
\right.
$$
for $k,n\ge 1$ and $j=1, \ldots , d$. Let $\mathbb{F}^{k,j} := \{ \mathcal{F}^{k,j}_t; t\ge 0 \} $ be the natural filtration generated by $\{A^{k,j}(t);  t \ge 0\}$. The multi-dimensional filtration generated by $A^k$ is naturally characterized as follows. Let $\mathbb{F}^{k} := \{\mathcal{F}^{k}_t ; 0 \leq t <\infty\}$ be the product filtration given by $\mathcal{F}^{k}_t := \mathcal{F}^{k,1}_t \otimes\mathcal{F}^{k,2}_t\otimes\cdots\otimes\mathcal{F}^{k,d}_t$ for $t\ge 0$. Let $\mathcal{T}:=\{T^{k}_m; k,m\ge 0\}$ be the order statistics obtained from the family of random variables $\{T^{k,j}_\ell; \ell,k\ge 0 ;j=1,\ldots,d\}$. That is, we set $T^{k}_0:=0$,

%\label{difst}
\begin{equation}\label{hittingtimes2}
T^{k}_1:= \inf_{1\le j\le d}\Big\{T^{k,j}_1 \Big\},\quad T^{k}_n:= \inf_{\substack {1\le j\le d\\ m\ge 1} } \Big\{T^{k,j}_m ; T^{k,j}_m \ge  T^{k}_{n-1}\Big\}
\end{equation}
for $n\ge 1$.

The structure $\mathscr{D} :=\{\mathcal{T}, A^{k,j}; 1\le j\le d, k\ge 1\}$ is a \textit{discrete-type skeleton} for the Brownian motion in the language of \cite{LEAO_OHASHI2017.1} (see Def. 2.1 in \cite{LEAO_OHASHI2017.1}). Throughout this article, we set

\begin{equation}\label{Nk}
\Delta T^{k}_n: = T^{k}_n - T^{k}_{n-1},\quad N^{k}(t):=\max\{n; T^{k}_n\le t\},\quad \bar{t}_k:=\max\{T^k_n; T^k_n \le t\}; t\ge 0,
\end{equation}
and
$$
\eta^{k,j}_n:=\left\{
\begin{array}{rl}
1; & \hbox{if} \  \Delta A^{k,j} (T^k_n)>0 \\
-1;& \hbox{if} \  \Delta A^{k,j} (T^k_n)< 0 \\
0;& \hbox{if} \ \Delta A^{k,j} (T^k_n)=0,
\end{array}
\right.
$$
for $n\ge 1$. Let us denote $\eta^k_n:=\big(\eta^{k,1}_n, \dots, \eta^{k,d}_n\big)$,
$$\mathbb{I}_k:=\Big\{ (i^k_1, \ldots, i^k_d); i^k_\ell\in \{-1,0,1\}~\forall \ell \in \{1,\ldots, d\}~\text{and}~\sum_{j=1}^d|i^k_j|=1   \Big\}$$
and $\mathbb{S}_k:=(0,+\infty)\times \mathbb{I}_k$ for $k,n\ge 1$. The $n$-fold Cartesian product of $\mathbb{S}_k$ is denoted by $\mathbb{S}_k^n$ and a generic element of $\mathbb{S}^n_k$ will be denoted by $\mathbf{b}^k_n := (s^k_1,\tilde{i}^k_1, \ldots, s^k_n, \tilde{i}^k_n)\in \mathbb{S}^n_k$ where $(s^k_r,\tilde{i}^k_r)\in (0,+\infty)\times \mathbb{I}_k$ for $1\le r\le n$. We also denote

$$t^k_j: = \sum_{\ell=1}^j s^k_\ell,$$
for a given list $(s^k_1, \ldots, s^k_j)\in (0,+\infty)^j$ where $j\ge 1$. The driving noise in our methodology is given by the following discrete-time process

\begin{equation}\label{Aknvariable}
\mathcal{A}^k_n:= \Big(\Delta T^k_1, \eta^k_1, \ldots, \Delta T^k_n, \eta^k_n\Big)\in \mathbb{S}^n_k~a.s.
\end{equation}
One should notice that $\mathcal{F}^k_{T^k_n} = (\mathcal{A}^k_n)^{-1}(\mathcal{B}(\mathbb{S}^n_k))$ up to $\mathbb{P}$-null sets, where $\mathcal{B}(\mathbb{S}^k_n)$ is the Borel $\sigma$-algebra generated by $\mathbb{S}^n_k; n\ge 1$.

%For simplicity of notation, we set $\epsilon_k = 2^{-k}; k\ge 1$.

\

\noindent \textbf{Transition Probabilities.} The law of the system will evolve according to the following probability measure defined on

$$\mathbb{P}^k_n(E):=\mathbb{P}\{\mathcal{A}^k_n\in E\}; E\in \mathcal{B}(\mathbb{S}^n_k), n\ge 1.$$
 By the very definition, $\mathbb{P}^k_{n}(\cdot) = \mathbb{P}^k_{r}(\cdot\times \mathbb{S}^{r-n}_k)$ for any $r> n\ge 1$. By construction $\mathbb{P}^k_{r}(\mathbb{S}^{n}_k\times \cdot)$ is a regular measure and $\mathcal{B}(\mathbb{S}_k)$ is countably generated, then it is well-known there exists ($\mathbb{P}^k_{n}$-a.s. unique) a disintegration $\nu^k_{n,r}: \mathcal{B}(\mathbb{S}^{r-n}_k)\times\mathbb{S}^{n}_k\rightarrow[0,1]$ which realizes

$$\mathbb{P}^k_{r}(D) = \int_{\mathbb{S}^{n}_k}\int_{\mathbb{S}^{r-n}_k} 1\!\!1_{D}(\textbf{b}^k_{n},q^k_{n,r})\nu^k_{n,r} (dq^k_{n,r}|\textbf{b}^k_{n})\mathbb{P}^k_{n}(d\textbf{b}^k_{n}),$$
for every $D\in \mathcal{B}(\mathbb{S}^{r}_k)$, where $q^k_{n,r}$ is the projection of $\textbf{b}^k_r$ onto the last $(r-n)$ components, i.e., $q^k_{n,r} = (s^k_{n+1},\tilde{i}^k_{n+1}, \ldots,s^k_{r},\tilde{i}^k_{r} )$ for a list $\textbf{b}^k_r = (s^k_1,\tilde{i}^k_1, \ldots, s^k_r,\tilde{i}^k_r)\in \mathbb{S}^r_k$. Sometimes, we denote $\nu^k_{0,r}:=\mathbb{P}^k_r$ for $r\ge 1$. If $r=n+1$, we denote $\nu^k_{n+1}:=\nu^k_{n,n+1}$. For an explicit formula for this conditional probability, we refer the reader to \cite{LEAO_OHASHI2019}.

%By the very definition, for each $E\in \mathcal{B}(\mathbb{S}_k)$ and $\mathbf{b}^k_{n}\in \mathbb{S}_k^{n}$, we have
%\begin{equation}\label{form1dis}
%\nu^k_{n+1}(E|\mathbf{b}^k_{n})= \mathbb{P}\Big\{(\Delta T^k_{n+1}, \eta^k_{n+1})\in E|\mathcal{A}^k_{n} = \mathbf{b}^k_{n}\Big\}; n\ge 1.
%\end{equation}

Let $\textbf{B}^p(\mathbb{F})$ be the space of c\`adl\`ag $\mathbb{F}$-adapted processes $Y$ such that

$$\|Y\|^p_{\textbf{B}^p}:=\mathbb{E}\| Y\|^p_{\infty} < \infty$$
where $\|Y\|_\infty:=\sup_{0\le t\le T}|Y(t)|$, $1\le p< \infty$ and $0 < T < +\infty$ is a fixed terminal time. In the sequel, $[\cdot, \cdot]$ is the $\mathbb{F}^k$-quadratic variation operator. The following class of $\mathbb{F}^k$-adapted processes whose the underlying differential structure and asymptotics were studied by \cite{LEAO_OHASHI2017.1} will play a key role in this work.

\begin{definition}\label{GASdef}
We say that $\mathcal{Y} = \big((X^k)_{k\ge 1},\mathscr{D}\big)$ is an \textbf{imbedded discrete structure} for $X\in \mathbf{B}^p(\mathbb{F})$ if $X^k$ is a sequence of $\mathbb{F}^k$-adapted pure jump processes of the form

\begin{equation}\label{purejumpmodel}
X^k(t) = \sum_{n=0}^\infty X^{k}(T^k_n)1\!\!1_{\{T^k_n\le t < T^k_{n+1}\}}; 0\le t\le T,
\end{equation}
it has integrable quadratic variation $\mathbb{E}[X^k,X^k](T)< \infty; k\ge 1$, and

\begin{equation}\label{scdef}
\lim_{k\rightarrow+\infty}\|X^k-X\|_{\textbf{B}^p}=0
\end{equation}
for some $p\ge 1$.
\end{definition}

\section{Construction of an imbedded structure for $S$ and near optimal stopping times}\label{varsection}
For $t\le T$, we denote $\mathcal{T}_t(\mathbb{F})$ as the set of all $\mathbb{F}$-stopping times $\tau$ such that $t\le \tau\le T$~a.s.. For $n\geq 0$, we denote by $\mathcal{T}_{k,n}(\mathbb{F}):=\mathcal{T}_{t}(\mathbb{F})$ for $t=T^k_n$. To shorten notation, we set $\mathcal{T}_{k,n}:=\mathcal{T}_{k,n}(\mathbb{F})$. Throughout this article, we assume that reward process $Z$ is an $\mathbb{F}$-adapted continuous process and it satisfies the integrability regularity condition

\

\noindent \textbf{(A1)} $\|Z\|^p_{\textbf{B}^p}< \infty~\forall p\ge 1.$

\

For a given reward process $Z$, let $S$ be the Snell envelope associated with $Z$

\[
S (t):= \text{ess} \sup_{\tau\in \mathcal{T}_t(\mathbb{F})} \mathbb{E} \left[ Z(\tau)  \mid \mathcal{F}_t \right], \quad 0 \leq t \leq T.
\]
We assume $S$ satisfies the following integrability condition:

\

\noindent \textbf{(A2)} $\|S\|^p_{\textbf{B}^p}< \infty~\forall p\ge 1.$

\

Assumptions \textbf{(A1-A2)} are not essential in the sense that the range of integrability can be relaxed. In order to simplify the exposition, we then assume \textbf{(A1-A2)} hold true. Since the optimal stopping time problem at hand takes place on the compact set $[0,T]$, it is crucial to know the correct number of periods in our discretization scheme. For this purpose, let $\lceil x\rceil$ be the smallest natural number bigger or equal to $x\ge 0$. We then denote

$$e(k,T):= d\lceil \epsilon^{-2}_k T\rceil; k\ge 1.$$

\begin{lemma}\label{ratetenk}
For any $0 < T <\infty$,
$$\mathbb{E}|T^{k}_{e(k,T)}- T|^2 \rightarrow 0$$
and $T^k_{e(k,T)}\rightarrow T$~a.s. as $k\rightarrow \infty$.
\end{lemma}
\begin{proof}
Just notice that

$$
|T^k_{d\lceil \epsilon_k^{-2}T \rceil} - T|\le \max\Big\{|\max_{1\le i\le d} T^{k,i}_{\lceil \epsilon_k^{-2}T \rceil} - T|, |\min_{1\le i\le d}T^{k,i}_{\lceil \epsilon_k^{-2}T \rceil} - T|   \Big\}~a.s.
$$
for every $k\ge 1$. We then apply Lemma 2.2 in \cite{koshnevisan} to conclude the proof.
\end{proof}
Due to this result, we will reduce the analysis to the deterministic number of periods $e(k,T)$. We denote $D^{k,m}_n$ as the set of all $\mathbb{F}^k$-stopping times of the form

$$
\tau = \sum_{i=n}^{m}T^k_i1\!\!1_{\{\tau = T^k_i\}},
$$
where $\{\tau = T^k_i; i=n, \ldots, m\}$ is a partition of $\Omega$ and $0\le n \le m$. Let us denote $D^k_{n,T} := \{\eta\wedge T; \eta\in D^{k,\infty}_n\}$. We observe that $D^{k,\infty}_0$ is the set of all $\mathbb{F}^k$-totally inaccessible stopping times.
%, see Th 5.62 in \cite{he}.

Let $\{Z^k; k\ge 1\}$ be a sequence of pure jump processes of the form (\ref{purejumpmodel}) and let $\{S^k; k\ge 1\}$ be the associated value process given by

\begin{equation}\label{discretevaluep}
S^k(t) := \sum_{n=0}^{\infty} S^k(T^k_n)\mathds{1}_{\{T^k_n\le t \wedge T^k_{e(k,T)}< T^k_{n+1}\}}; 0\le t\le T,
\end{equation}
where for $0\le n\le e(k,T)$, we denote

$$S^k(T^k_n):= \esssup_{\tau\in D^{k,e(k,T)}_{n}}\mathbb{E}\Big[ Z^k(\tau\wedge T)\big|\mathcal{F}^k_{T^k_n}\Big]\quad \text{and}\quad U^{\mathcal{Y},k,\textbf{p}}S(T^k_n):=\mathbb{E}\Bigg[\frac{\Delta S^{k}(T^k_{n+1})}{\epsilon^2_k}\Big|\mathcal{F}^k_{T^k_n}\Bigg].
$$

The following result shows the fundamental role played by the imbedded discrete structure driven by a skeleton $\mathscr{D} :=\{\mathcal{T}, A^{k,j}; 1\le j\le d, k\ge 1\}$: It enables to lift the discrete structure into the Brownian filtration without loosing $\mathbb{F}^k$-adaptedness in the optimization problem. We postpone the proof of Proposition \ref{equality} to Section \ref{sectionproof}.

\begin{proposition}\label{equality}
Let $Z^k$ be a pure jump process of the form (\ref{purejumpmodel}). For each $n\ge 0$ and $k\ge 1$, we have

$$
\esssup_{\tau\in D^{k}_{n,T}}\mathbb{E}\Big[Z^k\big(\tau\big)|\mathcal{F}^k_{T^k_n}\Big] = \esssup_{\tau\in \mathcal{T}_{k,n}}\mathbb{E}\Big[Z^k\big(\tau\big)|\mathcal{F}^k_{T^k_n}\Big]~a.s.
$$
\end{proposition}

\subsection{$\epsilon$-optimal stopping times}\label{constructionvalues} The dynamic programming algorithm allows us to define the stopping and continuation regions as follows

$$\textbf{S}(i,k):=\Big\{\textbf{b}^k_{i}\in  \mathbb{S}^{i}_k; \mathbb{Z}^k_i(\textbf{b}^k_i) = \mathbb{V}^k_i(\textbf{b}^k_i) \Big\}\quad \text{(stopping region)}$$

$$\textbf{D}(i,k):=\Big\{\textbf{b}^k_{i}\in  \mathbb{S}^{i}_k; \mathbb{V}^k_i(\textbf{b}^k_i) > \mathbb{Z}^k_i(\textbf{b}^k_i) \Big\}\quad \text{(continuation region)}$$
where $0\le i\le e(k,T)$ and $\mathbb{V}^k_n:\mathbb{S}^n_k\rightarrow\mathbb{R}$ and $\mathbb{Z}^k_n:\mathbb{S}^n_k\rightarrow\mathbb{R}$ are Borel functions which realize

\begin{equation}\label{Zbb}
S^{k}(T^k_n) = \mathbb{V}^k_n(\mathcal{A}^k_n)~a.s\quad \text{and}~Z^k(T^k_n\wedge T) = \mathbb{Z}^k_n(\mathcal{A}^k_n)~a.s; n=0, \ldots, e(k,T).
\end{equation}
Let $Y^k(i):=Z^k(T^k_i\wedge T); i\ge 0.$ and we set $J^{k,m}_n$ as the set of all $(\mathcal{F}^k_{T^k_i})^m_{i=0}$-stopping times taking values on $\{n,n+1, \ldots, m\}$ for a given $0\le n < m < \infty$. Following the classical theory of discrete optimal stopping, the smallest $(\mathcal{F}^k_{T^k_i})^{e(k,T)}_{i=0}$-optimal stopping-time w.r.t the problem

$$\sup_{\tau \in J^{e(k,T)}_0}\mathbb{E} \big[Y^k(\tau)\big]$$
is given by
\begin{equation}\label{bbbb}
\tau^{k}:=\min \Big\{0\le j\le e(k,T); \mathcal{A}^k_{j}\in \textbf{S}(j,k)\Big\}=\min\Big\{0\le j\le e(k,T); S^{k}(T^k_j) = Z^k(T^k_j \wedge T)\Big\}
\end{equation}
which is finite a.s. by construction. Moreover,

\begin{equation}\label{is}
\text{ess}~\sup_{\eta\in J^{k,e(k,T)}_n}\mathbb{E}\Big[Y^k(\eta)|\mathcal{F}^{k}_{T^k_n}\Big] =\text{ess}~\sup_{\tau\in D^{k,e(k,T)}_n}\mathbb{E}\Big[Z^k(\tau\wedge T)|\mathcal{F}^{k}_{T^k_n}\Big]~a.s.
\end{equation}
for each $0\le n \le e(k,T)$. The dynamic programming principle can be written as

\begin{equation}\label{DPST}
\left\{\begin{array}{l}
 \tau^{k}_{e(k,T)}:= e(k,T) \\
\tau^{k}_{j}:= j 1\!\!1_{G^k_j} + \tau^{k}_{j+1}1\!\!1_{(G^{k}_j)^c}; 0\le j \le e(k,T)-1
\end{array}\right.
\end{equation}
where

$$G^k_j: = \Bigg\{\mathbb{Z}^k_j (\mathcal{A}^k_j)\ge \mathbb{E}\Big[\mathbb{Z}^k_{\tau^{k}_{j+1}}(\mathcal{A}^k_{\tau^{k}_{j+1}})\big|\mathcal{A}^k_j\Big]\Bigg\}; 0\le j\le e(k,T)-1$$
and $\tau^{k}=\tau^{k}_0$ a.s. The sequence of functions $\mathbf{U}^k_j:\mathbb{S}^j_k\rightarrow\mathbb{R}$

\begin{equation}\label{cvalues}
\textbf{b}^k_j\mapsto \mathbf{U}^k_j(\textbf{b}^k_j):=\mathbb{E}\Big[\mathbb{Z}^k_{\tau^{k}_{j+1}}(\mathcal{A}^k_{\tau^{k}_{j+1}})\big|\mathcal{A}^k_j=\textbf{b}^k_j\Big]; 0\le j\le e(k,T)-1
\end{equation}
are called \textit{continuation values} and they play a key role in the obtention of the optimal value.

The value functional which gives the best payoff can be reconstructed by means of the dynamic programming principle over the $e(k,T)$-steps which provides

$$
\sup_{\eta\in J^{k,e(k,T)}_0}\mathbb{E}\big[Y^k(\eta)\big] =\max\Big\{\mathbb{Z}^k_0(\textbf{0}); \mathbb{E}\big[\mathbb{V}^k_1 (\mathcal{A}^k_1)\big]\Big\},
$$
where $\mathbb{E}\big[\mathbb{V}^k_1 (\mathcal{A}^k_1)\big] = \mathbb{E}\Big[\mathbb{Z}^k_{\tau^{k}_1}(\mathcal{A}^k_{\tau^{k}_1})\Big]$. Moreover,

\begin{equation}\label{aaaa}
\mathbb{E}\big[Y^k(\tau^k)\big]=\mathbb{E}\big[Z^k(T^k_{\tau^{k}}\wedge T)\big] = \sup_{\tau\in D^k_{0,T}} \mathbb{E}\big[Z^k(\tau \wedge T^k_{e(k,T)})\big]=\sup_{\tau\in \mathcal{T}_{0}(\mathbb{F})}\mathbb{E} \big[Z^k(\tau \wedge T^k_{e(k,T)})\big]
\end{equation}
where the last fundamental equality in identity (\ref{aaaa}) is due to Proposition \ref{equality}. Let us denote

\begin{equation}\label{valuef}
V^k_0:= \sup_{\tau\in D^k_{0,T}}\mathbb{E}\big[Z^k(\tau\wedge T^k_{e(k,T)})\big]; k\ge 1.
\end{equation}
We are now able to state the main results of this article. The proof of Theorem \ref{mainTh} is postponed to Section \ref{sectionproof}.

\begin{theorem}\label{mainTh}
Let $S$ be the Snell envelope associated to a reward process $Z$ satisfying assumptions \textbf{(A1-A2)}. Let $\{Z^k; k\ge 1\}$ be a sequence of pure jump processes of the form (\ref{purejumpmodel}) and let $\{S^k; k\ge 1\}$ be the associated value process given by (\ref{discretevaluep}). If $\mathcal{Z} = \big( (Z^k)_{k\ge 1},\mathscr{D}\big)$ is an imbedded discrete structure for $Z$ where (\ref{scdef}) holds for $p>1$, then $\mathcal{S} = \big((S^k)_{k\ge 1},\mathscr{D}\big)$ is an imbedded discrete structure for $S$, where

$$\lim_{k\rightarrow+\infty}\mathbb{E}\sup_{0\le t\le T}|S^k(t) - S(t)|=0.$$
Moreover, $ \{S^k; k\ge 1\}$ is the unique sequence of pure jump processes of the form (\ref{purejumpmodel}) which satisfies the following variational inequality

\begin{eqnarray}\label{varineq}
\max \Big\{ U^{\mathcal{Y},k,\textbf{p}}S(T^k_i);  Z^k(T^k_i\wedge T) - S^{k}(T^k_i) \Big\} & = & 0\quad i=e(k,T)-1, \ldots, 0,~a.s. \\
\nonumber S^{k} (T^k_{e(k,T)}) &=&Z^k (T^k_{e(k,T)}\wedge T)~a.s.
\end{eqnarray}
\end{theorem}

\begin{theorem}\label{optSTres}
Let $\mathcal{Z} = \big((Z^k)_{k\ge 1},\mathscr{D}\big)$ be an imbedded discrete structure for the reward process $Z$ and let $\tau^k$ be the optimal stopping time given by (\ref{bbbb}). Then, $T^k_{\tau^{k}}\wedge T$ is an $\epsilon$-optimal stopping time in the Brownian filtration, i.e., for a given $\epsilon>0$,

$$\sup_{\eta\in \mathcal{T}_0(\mathbb{F})}\mathbb{E}\big[Z(\eta)\big] -\epsilon< \mathbb{E}\big[Z(T^k_{\tau^k}\wedge T)\big]$$
for every $k$ sufficiently large. Moreover,

\begin{equation}\label{abcov}
\Big|\sup_{\tau\in \mathcal{T}_0(\mathbb{F})}\mathbb{E}\big[Z(\tau)\big] -  V^k_0\Big|\le \mathbb{E}\|Z^k(\cdot\wedge T^k_{e(k,T)}) - Z\|_\infty\rightarrow 0
\end{equation}
as $k\rightarrow+\infty$.
\end{theorem}

\begin{proof}
The imbedded discrete structure property, the path-continuity of $Z$ and Lemma \ref{ratetenk} yield
$$
\mathbb{E}\| Z^k(\cdot \wedge T^k_{e(k,T)}) - Z\|_\infty\rightarrow 0
$$
as $k\rightarrow +\infty$. This shows that

\begin{equation}\label{caudacov}
\Big|\sup_{\tau\in \mathcal{T}_0(\mathbb{F})}\mathbb{E}\big[Z^k(\tau \wedge T^k_{e(k,T)})\big] - \sup_{\tau\in \mathcal{T}_0(\mathbb{F})} \mathbb{E}\big[Z(\tau)\big]\Big|\le\mathbb{E}\sup_{0\le t\le T}|Z^k(t\wedge T^k_{e(k,T)}) - Z(t)|\rightarrow 0
\end{equation}
as $k\rightarrow +\infty$. By (\ref{aaaa}) (see also Proposition \ref{equality}), (\ref{valuef}) and (\ref{is}), we conclude (\ref{abcov}). Now, by using (\ref{aaaa}) and (\ref{caudacov}), for a given $\epsilon>0$, we have

\begin{equation}\label{epotimokgrande}
\sup_{\tau\in \mathcal{T}_0(\mathbb{F})}\mathbb{E}\big[Z(\tau)\big] - \epsilon< \sup_{\tau\in \mathcal{T}_0(\mathbb{F})} \mathbb{E}\big[Z^k(\tau \wedge T^k_{e(k,T)} \wedge T)\big] = \mathbb{E}\big[Z^k(T^k_{\tau^{k}}\wedge T)\big]
\end{equation}
for every $k$ sufficiently large. The imbedded discrete structure property of $Z^k$ yields $\mathbb{E}\big[Z^k(T^k_{\tau^{k}}\wedge T)\big] <\mathbb{E}\big[Z(T^k_{\tau^{k}}\wedge T)\big] + \epsilon  $ for every $k$ sufficiently large. This implies $\sup_{\tau\in\mathcal{T}_0(\mathbb{F})}\mathbb{E}\big[Z(\tau)\big] - 2\epsilon < \mathbb{E}\big[Z(T^k_{\tau^{k}}\wedge T)\big]$ for every $k$ sufficiently large. It remains to show $T^k_{\tau^k}$ is an $\mathbb{F}$-stopping time. Clearly,  $T^k_{\tau^k}:\Omega\rightarrow \mathbb{R}_+$ is $\mathcal{F}$-measurable. We claim that

$$T^k_{\tau^k(\omega)}(\omega)\le t~\text{and}~\omega|_{[0,t]} = \omega'|_{[0,t]}~\text{then}~T^k_{\tau^k(\omega)}(\omega) = T^k_{\tau^k(\omega')}(\omega').$$
 Recall $\tau^{k}(\omega)=\min \Big\{0\le j\le e(k,T); \mathcal{A}^k_{j}(\omega)\in \textbf{S}(j,k)\Big\}$ and we notice that $T^k_{\tau^k(\omega)}(\omega)\le t$ means that all information that we need to compute $T^k_{\tau^k(\omega)}(\omega)$ lies on the Brownian path $\omega|_{[0,t]}$. This can be easily seen by the fact that each continuous time random walk $A^{k,j}; j=1, \ldots, d$ only (possibly) jumps at the stopping times $(T^k_n)_{n\ge 0}$. In this case, if $\omega|_{[0,t]} = \omega'|_{[0,t]}$ and $T^k_{\tau^k(\omega)}(\omega)\le t$ , then we necessarily have $T^k_{\tau^k(\omega)}(\omega) = T^k_{\tau^k(\omega')}(\omega')$. From Galmarino's test (see \cite{dellacherie} Nos. IV 94-103, pp 145-152), we conclude that $T^k_{\tau^k}$ is an $(\widetilde{\mathcal{F}}_t)_{t\ge 0}$-stopping time, where $(\widetilde{\mathcal{F}}_t)_{t\ge 0}$ is the raw filtration generated by the Brownian motion on the Wiener space. This shows that $T^k_{\tau^k}\wedge T$ is an $\mathbb{F}$-stopping time and (\ref{epotimokgrande}) allows us to conclude the proof.
\end{proof}
\begin{remark}
%In typical non-Markovian examples (see Propositions \ref{exBM} and \ref{exFBM}), we are able to exhibit explicit rates of convergence for (\ref{abcov}) as a function of $\{\epsilon_k; k\ge 1\}$.
The importance of imbedded discrete structures is already apparent in Theorem \ref{optSTres}: Maximization along the discrete type filtration $(\mathcal{F}^k_{T^k_n})_{n=0}^{e(k,T)}$ is essentially equivalent to maximization along the Brownian filtration for $k$ large enough. More importantly, the imbedded structure allows us to reduce dimensionality of the optimal stopping problem (see Section \ref{appliedsec} and Example \ref{EXAMPLE_PAPER}).
\end{remark}

\section{Proofs of Proposition \ref{equality} and Theorem~\ref{mainTh}}\label{sectionproof}

In order to prove convergence, we need a subtle pathwise argument on the conditional expectations involved in the optimization problem. For a given pure jump process of the form (\ref{purejumpmodel}), there exists a list of Borel functions $h^k_\ell:\mathbb{S}_k^\ell\rightarrow \mathbb{R}$ which realizes

$$Z^k(t) = \sum_{\ell=0}^{\infty} h^k_\ell(\mathcal{A}^k_\ell)\mathds{1}_{\{T^k_\ell \le t < T^k_{\ell+1}\}}~a.s.; 0\le t\le T.$$
In other words, $\mathbb{Z}^k_n(\mathcal{A}^k_n) = Z^k(T^k_n\wedge T)$~a.s., where

$$
\mathbb{Z}^k_n(\textbf{b}^k_n)=\left\{
\begin{array}{rl}
h^k_n(\textbf{b}^k_n); & \hbox{if} \ t^k_n\le T \\
h^k_j(\textbf{b}^k_j);& \hbox{if} \ T < t^k_n, t^k_j\le T < t^k_{j+1} \\
\end{array}
\right.
$$
for $n\ge 0$. In the sequel, we fix two natural numbers $0\le n \le r$. For a given $\tau\in D^{k,r}_n$, there exists a list of Borel functions $g^k_\ell:\mathbb{S}_k^\ell\rightarrow \mathbb{R}$ which realizes

\begin{equation}\label{controlrep}
\tau = \sum_{\ell=n}^r T^k_\ell g^k_\ell(\mathcal{A}^k_\ell)~a.s
\end{equation}
where $g^k_\ell(\mathcal{A}^k_\ell) = \mathds{1}_{\{\tau = T^k_\ell\}}; n\le \ell \le r$. The action space in our optimization problem is given by

$$\mathbb{A}^j:=\Big\{x = (x_1,\ldots, x_j)\in \mathbb{N}^j; \sum_{\ell=1}^j x_\ell = 1\Big\}; j>1.$$
The elements of $\mathbb{A}^{r-n+1}\times \mathbb{S}^r_k$ will be denoted by

$$\textbf{o}^{k,n,r}:=\big(a^k_n, \ldots, a^k_r, \textbf{b}^k_r\big).$$

For a given list of Borel functions $g^k_\ell:\mathbb{S}_k^\ell\rightarrow \{0,1\}$ realizing (\ref{controlrep}), we define the map $\Xi^{k,g^k}_{n,r}:\mathbb{S}^r_k\rightarrow \mathbb{A}^{r-n +1}\times \mathbb{S}^r_k$ given by

$$\Xi^{k,g^k}_{n,r}(\textbf{b}^k_r):= \Big( g^k_n(\textbf{b}^k_n), \ldots, g^k_r(\textbf{b}^k_r), \textbf{b}^k_r\Big);~\textbf{b}^k_r\in \mathbb{S}^r_k.$$
We observe that we may choose $(g^k_\ell)_{\ell=n}^r$ in such way that $(g^k_n(\textbf{b}^k_n), \ldots, g^k_r(\textbf{b}^k_r))\in\mathbb{A}^{r-n+1}$ for every $\textbf{b}^k_r\in \mathbb{S}^r_k$. Moreover, there exists an explicit map $\gamma^k_{n,r}:\mathbb{A}^{r-n+1}\times \mathbb{S}^r_k\rightarrow\mathbb{R}$ such that

\begin{equation}\label{prepZk}
Z^k(\tau \wedge T) = \gamma^k_{n,r}\Big( \Xi^{k,g^k}_{n,r}\big(\mathcal{A}^k_r\big) \Big)~a.s.,
\end{equation}
where

$$
\gamma^k_{n,r}(\textbf{o}^{k,n,r}):= \sum_{\ell=n}^r h^k_\ell(\textbf{b}^k_\ell)\mathds{1}_{\{1\}}(a^k_\ell)\mathds{1}_{\{t^k_\ell \le T\}}(\textbf{b}^k_\ell) +\sum_{\ell=n}^r  h^k_\ell(\textbf{b}^k_\ell)\mathds{1}_{\{t^k_\ell\le T < t^k_{\ell+1}\}}\Bigg(\sum_{j=n}^r\mathds{1}_{\{1\}}(a^k_j)\mathds{1}_{\{t^k_j > T\}}(\textbf{b}^k_j)\Bigg).
$$

We should also compute a pathwise representation of $Z^k(\tau\wedge T)$ but for a slightly more general class of stopping times. Let $\widetilde{\mathcal{T}}_{k,n,r}$ be the set of all $\mathbb{F}$-stopping times of the form

$$\eta = \sum_{\ell=n}^r T^k_{\ell}\mathds{1}_{\{\eta = T^k_\ell\}}$$
where $\{\eta = T^k_\ell\}\in \mathcal{F}_{T^k_\ell}; n\le \ell\le r$. Of course, $D^{k,r}_n\subset \widetilde{\mathcal{T}}_{k,n,r}$ for every $0\le n \le r $ and $k\ge 1$. Recall that $T^k_n < \infty$ a.s for every $n\ge 0$. In this case, it is known (see e.g. Corollary 3.22 in \cite{he}) that $(\Phi^k_j)^{-1}\big(\mathcal{O}\big) = \mathcal{F}_{T^k_j}$, where $\mathcal{O}$ is the optional $\sigma$-algebra on $\Omega\times \mathbb{R}_+$ and

$$\Phi^k_j(\omega): = \big(\omega, T^k_j(\omega)\big); \omega\in \Omega^*,~j\ge 1,$$
where $\mathbb{P}(\Omega^*)=1$. To keep notation simple, we choose a version of $\Phi^k_j$ defined everywhere and with a slight abuse of notation we write it as $\Phi^k_j$. Based on this fact, for a given $\eta\in\widetilde{\mathcal{T}}_{k,n,r} $ there exists a list a Borel functions $\varphi^k_\ell:\Omega\times\mathbb{R}_+\rightarrow \{0,1\}$ which realizes

\begin{equation}\label{controlrepBM}
\eta = \sum_{\ell=n}^r T^k_\ell \varphi^k_\ell (\Phi^k_\ell)~a.s.
\end{equation}
For a list of Borel functions $\varphi^k_\ell:\Omega\times\mathbb{R}_+\rightarrow\{0,1\}$ realizing (\ref{controlrepBM}), we then define the map $\Xi^{k,\varphi^k}_{n,r}:\Omega\times \mathbb{R}^{r-n+1}_+\rightarrow\mathbb{A}^{r-n+1}\times \mathbb{S}^r_k$ given by

$$\Xi^{k,\varphi^k}_{n,r}(\omega, x_n, \ldots, x_r,\textbf{b}^k_r):= \Big( \varphi^k_n\big(\Phi^k_n(\omega,x_n)\big), \ldots, \varphi^k_r\big(\Phi^k_r(\omega,x_r)\big), \textbf{b}^k_r\Big),$$
for $\omega\in \Omega, (x_n, \ldots, x_r)\in \mathbb{R}^{r-n+1}_+$ and $\textbf{b}^k_r\in \mathbb{S}^r_k$. By construction, we have

\begin{equation}\label{prepZkBM}
Z^k(\eta\wedge T) = \gamma^k_{n,r}\Big(\Xi^{k,\varphi^k}_{n,r}\big( J^k_{n,r}\big)  \Big)~a.s.,
\end{equation}
where $J^k_{n,r}:=\big(\text{Id}, T^k_n, \ldots, T^k_r, \mathcal{A}^k_r\big)$ and $\text{Id}:\Omega\rightarrow\Omega$ is the identity map.

In the sequel, $H^k_{n,r}:\mathcal{B}(\Omega\times \mathbb{R}^{r-n+1}_{+})\times\mathbb{S}^r_k\rightarrow[0,1]$ is the disintegration of $\mathbb{P}\circ J^k_{n,r}$ w.r.t $\mathbb{P}^k_r$ and $\nu^k_{n,r}$ is the disintegration of $\mathbb{P}^k_r$ w.r.t $\mathbb{P}^k_n$ for $r > n\ge 1$. Recall the notation introduced in Section \ref{preliminaries}.

%$q^k_{n,r}$ denotes the projection of $\textbf{b}^k_r$ onto the last $(r-n)$ components, i.e., $q^k_{n,r} = (s^k_{n+1},\tilde{i}^k_{n+1}, \ldots,s^k_{r},\tilde{i}^k_{r} )$ for a list $\textbf{b}^k_r = (s^k_1,\tilde{i}^k_1, \ldots, s^k_r,\tilde{i}^k_r)\in \mathbb{S}^r_k$. We also set $\nu^k_{0,r}:=\mathbb{P}^k_r$ for $r\ge 1$.

\begin{lemma}\label{condformulas}
Let $Z^k$ be a pure jump process of the form (\ref{purejumpmodel}). For each $\tau\in D^{k,r}_n$ and $\eta\in \widetilde{\mathcal{T}}_{k,n,r}$, for $0\le n < r$, we have

$$
\mathbb{E}\Big[Z^k(\tau \wedge T)|\mathcal{F}^k_{T^k_n}\Big] =\int_{\mathbb{S}^j_k}\gamma^k_{n,r}\Big(\Xi^{k,g^k}_{n,r}\big(\mathcal{A}^k_n, q^k_{n,r}\big)\Big)\nu^k_{n,r}(dq^k_{n,r}|\mathcal{A}^k_n)~a.s
$$
and

$$
\mathbb{E}\Big[Z^k(\eta \wedge T)|\mathcal{F}^k_{T^k_n}\Big]=\int_{\mathbb{S}^j_k}\int_{\Omega\times \mathbb{R}^j_{+}}\gamma^k_{n,r}\Big(\Xi^{k,\varphi^k}_{n,r}\big(\omega,x_n,\ldots,x_r,\mathcal{A}^k_n, q^k_{n,r}\big)\Big)H^k_{n,r}(d\omega dx_n\ldots dx_r|\mathcal{A}^k_n, q^k_{n,r})\nu^k_{n,r}(dq^k_{n,r}|\mathcal{A}^k_n),
$$
a.s for $j=r-n$, where $(g^k_\ell)_{\ell=n}^r$ and $(\varphi^k_\ell)_{\ell=n}^r$ are Borel functions associated to $\tau$ and $\eta$, respectively.
\end{lemma}
\begin{proof}
The proof is a direct consequence of the representations (\ref{prepZk}) and (\ref{prepZkBM}), so we omit the details.
\end{proof}

We are now able to prove Proposition \ref{equality} which is the main ingredient in the proof of Theorem \ref{mainTh} and the estimate (\ref{abcov}) in Theorem \ref{optSTres}.

\

\textbf{Proof of Proposition \ref{equality}}. For the moment, let us fix $r > n$. We claim

\begin{equation}\label{rest1}
\esssup_{\tau\in \widetilde{\mathcal{T}}_{k,n,r} }\mathbb{E}\Big[Z^k\big(\tau\wedge T\big)|\mathcal{F}^k_{T^k_n}\Big] = \esssup_{\tau\in \overline{\mathcal{T}}_{k,n,r}}\mathbb{E}\Big[Z^k\big(\tau\wedge T\big)|\mathcal{F}^k_{T^k_n}\Big]~a.s.,
\end{equation}
where $\overline{\mathcal{T}}_{k,n,r}:=\{\tau~\text{is an}~\mathbb{F}-~\text{stopping time};~T^k_n\le \tau \le T^k_r~a,s\}$. For a given $\tau\in \overline{\mathcal{T}}_{k,n,r}$, let us define the $\mathbb{F}$-stopping time $Q(\tau) = \min\{T^k_p; T^k_p > \tau\}$ and

$$\eta(\tau)  = \sum_{\ell=n}^{r+1}T^k_{\ell-1}\mathds{1}_{\{Q(\tau) = T^k_\ell\}}\in \widetilde{\mathcal{T}}_{k,n,r}.$$
We observe that $\mathbb{P}\{\eta(\tau) = T^k_{n-1}\}=0$. Moreover, since $Z^k$ has c\`adl\`ag paths, we have

$$
Z^k(\tau \wedge T) =Z^k\big((Q(\tau)\wedge T) -\big)=Z^k\big(\eta(\tau)\wedge T\big)~a.s.
$$
This proves (\ref{rest1}). We now claim that

\begin{equation}\label{rest2}
\esssup_{\tau\in \widetilde{\mathcal{T}}_{k,n,r} }\mathbb{E}\Big[Z^k\big(\tau\wedge T\big)|\mathcal{F}^k_{T^k_n}\Big] = \esssup_{\tau\in D^{k,r}_{n}}\mathbb{E}\Big[Z^k\big(\tau\wedge T\big)|\mathcal{F}^k_{T^k_n}\Big]~a.s.
\end{equation}
In order to check (\ref{rest2}), we make use of Lemma \ref{condformulas} as follows. Let us define the following objects:

$$F^k_{n,r}(\textbf{b}^k_n):=\max_{(a^k_n,\ldots,a^k_r)\in \mathbb{A}^{r-n+1}}G^k_{n,r}(a^k_n,\ldots,a^k_r,\textbf{b}^k_n),$$
where
$$G^{k}_{n,r}(a^k_n,\ldots,a^k_r, \textbf{b}^k_n):=\int_{\mathbb{S}^{r-n}_k}\gamma^k_{n,r}\Big(a^k_n,\ldots,a^k_r,\textbf{b}^k_n, q^k_{n,r}\big)\Big)\nu^k_{n,r}(dq^k_{n,r}|\textbf{b}^k_n),$$
and

$$\widetilde{F}^k_{n,r}(\textbf{b}^k_n):=\max_{(a^k_n,\ldots,a^k_r)\in \mathbb{A}^{r-n+1}}\widetilde{G}^k_{n,r}(a^k_n,\ldots,a^k_r, \textbf{b}^k_n),$$
where
$$\widetilde{G}^k_{n,r}(a^k_n,\ldots,a^k_r, \textbf{b}^k_n):=\int_{\mathbb{S}^{r-n}_k}\int_{\Omega\times \mathbb{R}^{r-n+1}_{+}}\gamma^k_{n,r}\Big(a^k_n,\ldots,a^k_r,\textbf{b}^k_n, q^k_{n,r}\big)\Big)H^k_{n,r}(d\omega dx_n\ldots dx_r|\textbf{b}^k_n, q^k_{n,r})\nu^k_{n,r}(dq^k_{n,r}|\textbf{b}^k_n),
$$
for each $\textbf{b}^k_n\in \mathbb{S}^k_n$ and $(a^k_n,\ldots,a^k_r)\in \mathbb{A}^{r-n+1}$. Since for each $E\in \mathcal{B}(\mathbb{S}^{r-n}_k))$ and $F\in\mathcal{B}(\Omega\times \mathbb{S}^{r-n+1}_k)$, the disintegrations

$$\textbf{b}^k_n\mapsto \nu^k_{n,r}(E|\textbf{b}^k_n)~\text{and}~ \textbf{b}^k_r\mapsto H^k_{n,r}(F|\textbf{b}^k_r)$$
are Borel functions, then we can safely state that (see Prop 7.29 in \cite{bertsekas}) both $G^k_{n,r}$ and $\widetilde{G}^k_{n,r}$ are Borel functions. Therefore, Lemma \ref{condformulas} yields

$$G^{k}_{n,r}(\text{Pr}_n\circ\Xi^{k,g^k}_{n,r}(\mathcal{A}^k_n)) = \mathbb{E}\Big[Z^k(\tau\wedge T)|\mathcal{F}^k_{T^k_n}\Big],\quad\widetilde{G}^k_{n,r}\big(\text{Pr}_n\circ\Xi^{k,\varphi^k}_{n,r}(\mathcal{A}^k_n)\big) = \mathbb{E}\Big[Z^k(\eta\wedge T)|\mathcal{F}^k_{T^k_n}\Big]$$
a.s for each $\tau\in D^k_{n,r}$ and $\eta\in \widetilde{\mathcal{T}}_{k,n,r}$, where

$$\text{Pr}_n\circ\Xi^{k,g^k}_{n,r}(\textbf{b}^k_n):=\Big( g^k_n(\textbf{b}^k_n), \ldots, g^k_r(\textbf{b}^k_r), \textbf{b}^k_n\Big)$$
is the projection of $\Xi^{k,g^k}_{n,r}(\textbf{b}^k_r)$ onto $\mathbb{A}^{r-n+1}\times \mathbb{S}^n_k$ and
$$\text{Pr}_n\circ\Xi^{k,\varphi^k}_{n,r}(\textbf{b}^k_n):=\Big( \varphi^k_n\big(\Phi^k_n(\omega,x_n)\big), \ldots, \varphi^k_r\big(\Phi^k_r(\omega,x_r)\big), \textbf{b}^k_n\Big)$$
is the projection of $\Xi^{k,\varphi^k}_{n,r}(\textbf{b}^k_r)$ onto $\mathbb{A}^{r-n+1}\times \mathbb{S}^n_k$. By construction,

$$
F^k_{n,r}(\mathcal{A}^k_n) = \esssup_{\tau\in D^{k,r}_{n}}\mathbb{E}\Big[Z^k\big(\tau\wedge T\big)|\mathcal{F}^k_{T^k_n}\Big]~a.s
$$
and
$$
\widetilde{F}^k_{n,r}(\mathcal{A}^k_n)=\esssup_{\tau\in \widetilde{\mathcal{T}}_{k,n,r} }\mathbb{E}\Big[Z^k\big(\tau\wedge T\big)|\mathcal{F}^k_{T^k_n}\Big]~a.s.
$$
we notice that $F^k_{n,r} = \widetilde{F}^k_{n,r}$ and hence the claim (\ref{rest2}) holds true. The argument outlined above actually shows that (\ref{rest2}) holds when $r=+\infty$, namely

\begin{equation}\label{rest4}
\esssup_{\tau\in \widetilde{\mathcal{T}}_{k,n} }\mathbb{E}\Big[Z^k\big(\tau\wedge T\big)|\mathcal{F}^k_{T^k_n}\Big] = \esssup_{\tau\in D^{k,\infty}_{n}}\mathbb{E}\Big[Z^k\big(\tau\wedge T\big)|\mathcal{F}^k_{T^k_n}\Big]~a.s
\end{equation}
where $\widetilde{\mathcal{T}}_{k,n}$ is the set of all $\mathbb{F}$-stopping times of the form

$$\tau = \sum_{\ell=n}^{\infty} T^k_\ell\mathds{1}_{\{\tau = T^k_\ell\}}$$
and $\{\tau = T^k_\ell\}\in \mathcal{F}_{T^k_\ell}; \ell\ge n$. This holds due to the fact that the set of sequences of natural numbers $(x_i)_{i=n}^{\infty}$ such that $\sum_{j=n}^{\infty}x_j = 1$ is countable. Lastly, the argument to prove (\ref{rest1}) actually shows that

\begin{equation}\label{rest5}
\esssup_{\tau\in \widetilde{\mathcal{T}}_{k,n} }\mathbb{E}\Big[Z^k\big(\tau\wedge T\big)|\mathcal{F}^k_{T^k_n}\Big] = \esssup_{\tau\in \overline{\mathcal{T}}_{k,n}}\mathbb{E}\Big[Z^k\big(\tau\wedge T\big)|\mathcal{F}^k_{T^k_n}\Big]~a.s.,
\end{equation}
where $\overline{\mathcal{T}}_{k,n}:=\{\tau~\text{is an}~\mathbb{F}-\text{stopping time};~T^k_n\le \tau \le +\infty~a.s.\}$. Identities (\ref{rest4}) and (\ref{rest5}) allow us to conclude the proof.

\

\noindent \textbf{Proof of Theorem~\ref{mainTh}:}
The variational inequality (\ref{varineq}) is a straightforward consequence of the classical discrete time dynamic programming principle so we shall omit this proof. Let us check the convergence. At first, we observe the Snell envelope $S$ has continuous paths. Indeed, by assumption \textbf{(A1)}, the reward process is a class D regular process and hence by applying Th 2.3.5 in \cite{lamberton} and assumption \textbf{(A2)}, the associated Snell envelope is a D regular process as well. Therefore, the Snell envelope $S$ has continuous paths under assumptions \textbf{(A1-A2)}. Let us define

$$\delta^k S(t): = \sum_{n=0}^{\infty}\mathbb{E}\Big[S(T^k_n)|\mathcal{F}^k_{T^k_n}\Big]\mathds{1}_{\{T^k_n\le t < T^k_{n+1}\}}; 0\le t\le T.$$
By Lemma 3.1 in \cite{LEAO_OHASHI2017.1} and assumption \textbf{(A2)}, we may use uniform integrability to safely state that

$$\lim_{k\rightarrow+\infty}\mathbb{E}\sup_{0\le t\le T}|\delta^kS(t) - S(t)|=0.$$
One can easily check that $\{\mathbb{E} \left[ Z(\tau) \mid \mathcal{F}_{T^k_n} \right]; \tau \in \mathcal{T}_{k,n}\}$ has the lattice property (see e.g. Def. 1.1.2 in \cite{lamberton}). Therefore, the tower property of the conditional expectation and Prop. 1.1.4 in \cite{lamberton} yield

$$
\delta^k S (t)=\sum_{n=0}^{\infty} \mathbb{E} \left[ S(T^k_n) \mid \mathcal{F}^k_{T^k_n} \right]1\!\!1_{ \{T^k_n \leq t < T^k_{n+1} \}  } = \sum_{n=0}^{\infty} \text{ess} \sup_{\tau \in \mathcal{T}_{k,n}}   \mathbb{E} \left[ Z(\tau) \mid  \mathcal{F}^k_{T^k_n} \right]1\!\!1_{ \{T^k_n \leq t < T^k_{n+1} \}}~a.s.,
$$
for $0 \leq t \leq T$. On the other hand, from Proposition \ref{equality}, we have

\begin{eqnarray*}
S^k(t)&=& \sum_{n=0}^{\infty} \text{ess} \sup_{\tau \in D^k_{n,T}}   \mathbb{E} \left[ Z^k (\tau\wedge T^k_{e(k,T)}) \mid  \mathcal{F}^k_{T^k_n} \right]1\!\!1_{ \{T^k_n \leq t < T^k_{n+1} \}  }\\
& &\\
&=&\sum_{n=0}^{\infty} \text{ess} \sup_{\tau \in \mathcal{T}_{k,n} }   \mathbb{E} \left[ Z^k (\tau\wedge T^k_{e(k,T)} ) \mid  \mathcal{F}^k_{T^k_n} \right]1\!\!1_{ \{T^k_n \leq t < T^k_{n+1} \} }~a.s.,
\end{eqnarray*}
for $0\le t\le T$. By applying Doob maximal inequality in the closable discrete-time martingale $\mathbb{E}[\sup_{0\le t\le T}|Z(t)-Z^k(t\wedge T^k_{e(k,T)})||\mathcal{F}^k_{T^k_n}]; n\ge 0$ and Jensen inequality, we can find a positive constant $C$ such that

\begin{eqnarray}
\nonumber\mathbb{E}\sup_{0\le t\le T}|S^k (t) - \delta^k S(t)|^p&\le& \mathbb{E}\sup_{n\ge 0}\Big| \mathbb{E}\Big[\sup_{0\le t\le T}|Z(t)-Z^k(t\wedge T^k_{e(k,T)})|\big|\mathcal{F}^k_{T^k_n}   \Big]\Big|^p 1\!\!1_{ \{T^k_n \leq T \}} \\
\nonumber& &\\
\label{sks}&\le& C \mathbb{E}\sup_{0\le t\le T}|Z(t)-Z^k(t\wedge T^k_{e(k,T)})|^p
\end{eqnarray}
as $k\rightarrow \infty$ for some $p>1$ due to the imbedded discrete property. In order to prove the right-hand side of (\ref{sks}) vanishes as $k\rightarrow\infty$ we just notice that $Z$ has continuous paths and a simple triangle inequality argument jointly with Lemma \ref{ratetenk} allows us to conclude the proof.

\section{Examples of non-Markovian optimal stopping problems}\label{ExamplesSECTION}
In this section, we show how Theorem \ref{optSTres} can be applied to concrete non-Markovian optimal stopping problems. To simplify the presentation, we set $d=1$. Throughout this section, we make use of the following notation. Let $D([0,t];\mathbb{R})$ be the linear space of $\mathbb{R}$-valued c\`adl\`ag paths on $[0,t]$ and let $\Lambda:=\{(t,\omega(\cdot \wedge t)); t\in [0,T]; \omega\in D([0,T];\mathbb{R})\}$ be the set of stopped paths endowed with the distance
%For each $f\in D([0,T];\mathbb{R})$, the image of $f$ at time $t$ is denoted by $f(t)$ and

%\begin{equation}\label{notation}
%f_t: = f(t\wedge \cdot).
%\end{equation}

$$d_{\beta}((t,\omega); (t',\omega')): = \sup_{0\le u\le T}|\omega(u\wedge t) - \omega'(u\wedge t')| + |t-t'|^{\beta},$$
where $0< \beta\le 1$. We say that $G$ is a \textit{non-anticipative} functional if it is a Borel mapping and

$$G(t,\omega) = G(t,\omega(\cdot \wedge t)); (t,\omega)\in[0,T]\times D([0,T];\mathbb{R}).$$

The coefficients of the SDEs will satisfy the following regularity conditions:

\

\noindent \textbf{Assumption I}: The non-anticipative mappings $\alpha: \Lambda \rightarrow \mathbb{R}$ and $\sigma:\Lambda \rightarrow \mathbb{R}$ are Lipschitz continuous, i.e., there exists a constant $K_{Lip}>0$ such that

$$|\alpha(t,\omega) - \alpha(t',\omega')| + |\sigma(t,\omega) - \sigma(t',\omega')|\le K_{Lip}d_{\theta} \big((t,\omega); (t',\omega')\big)$$
for every $t,t'\in [0,T]$ and $\omega,\omega'\in D([0,T];\mathbb{R})$, where $0 < \theta \le 1$.

\

\noindent \textbf{Assumption II}: The reward process is given by

\begin{equation}\label{rewardSDEBM}
Z(t) = F(t,X)
\end{equation}
where $X$ satisfies (\ref{pdsdeBM}) or (\ref{limsdefbm}) and $F:\Lambda \rightarrow\mathbb{R}$ is a non-anticipative Lipschitz functional, i.e., there exists constant $\|F\|$ such that

$$|F(t,\omega)- F(t',\omega')|\le \|F\|d_{\theta}\big((t,\omega);(t',\omega')\big)$$
for every $t,t'\in [0,T]$, $\omega,\omega'\in D([0,T];\mathbb{R})$, where $0 < \theta\le 1$.

\subsection{Non-Markovian SDE driven by Brownian motion} The underlying state process is the following SDE

\begin{equation}\label{pdsdeBM}
X(t) = x + \int_0^t \alpha(s,X)ds + \int_0^t \sigma(s,X)dB(s); 0\le t\le T,
\end{equation}
with a given initial condition $X(0)=x\in \mathbb{R}$. One can easily check by routine arguments that the SDE (\ref{pdsdeBM}) admits a strong solution in $\mathbf{B}^p(\mathbb{F})$ for every $p\ge 1$. The natural candidate for an imbedded discrete structure $\big((Z^k)_{k\ge1},\mathscr{D}\big)$ w.r.t $Z$ is given by

\begin{equation}\label{ZkBM}
Z^k(t) := \sum_{n=0}^{\infty}F\big(T^k_n, X^k\big)\mathds{1}_{\{T^k_n\le t < T^k_{n+1}\}}; 0\le t\le T
\end{equation}
where $X^k(0):=x$ and we define recursively

\begin{equation}\label{eulermaruyama}
X^{k}(T^k_{n}):=X^{k}(T^k_{n-1}) + \alpha\big(T^k_{n-1},X^{k}\big)\Delta T^k_{n} + \sigma (T^k_{n-1},X^k\big)\Delta A^{k}(T^k_{n})
\end{equation}
for $1\le n \le e(k,T)$ where $X^k(t)=\sum_{n=0}^{\infty} X^k(T^k_n)\mathds{1}_{\{T^k_n\le t \wedge T^k_{e(k,T)}< T^k_{n+1}\}}; 0\le t\le T.$

%In \cite{LEAO_OHASHI2017.2}, the authors examined convergence of $X^k$ in the more complex situation under the presence of controls. Then an almost direct application of Proposition 6.1 and Theorem 6.1 in \cite{LEAO_OHASHI2017.2} yields the following result:
%\mathbb{E}\sup_{0\le t\le T}|Z^k(t)  - Z(t)|^2
In the sequel, in order to simplify the presentation, we set $T=1$. Let $I^*$ be the Legendre transform of the hitting time $\inf\{t>0; |B(t)|=1\}$ (see e.g \cite{borodin}) given by

$$I^* (x) = \sup_{\lambda<0} \Bigg[ \lambda x - \ln \Bigg(\frac{1}{\text{cosh}(\sqrt{2|\lambda|})}\Bigg) \Bigg]; x<1.$$
\begin{proposition}\label{exBM}
If Assumptions \textbf{I-II} hold true for $\theta=1/2$ and $X$ satisfies (\ref{pdsdeBM}), then $\big((Z^k)_{k\ge 1},\mathscr{D}\big)$ given by (\ref{ZkBM}) is an imbedded discrete structure for $Z$. More importantly, for each $\beta \in (0,1)$ and $\zeta>1$, there exists a constant $C$ which depends on $\alpha$, $\sigma, \beta $ and $\zeta$ such that

%for every $\zeta>1$ and $\beta\in (0,1)$ there exists a constant $C$ which depends on $\alpha,\sigma, T, \zeta$ and $\beta \in (0,1)$ such that

\begin{equation}\label{rateBMex}
\big|V^k_0 - \sup_{\tau\in \mathcal{T}_0(\mathbb{F})}\mathbb{E}[Z(\tau)]\big|^2\le C\Big(  \epsilon^{2\beta}_k + \exp\big[-\zeta^{-1}\epsilon^{-2}_k I^* \big(1-\delta \big) \big] +  \delta \ln(2\delta^{-1})\Big)
\end{equation}

%$$
%\big|V^k_0 - \sup_{\tau\in \mathcal{T}_0(\mathbb{F})}\mathbb{E}[Z(\tau)]\big|^2\le C\Big(\sqrt{\epsilon_k}\ln\big(2 \epsilon^{-1/2}_k\big)\Big)
%$$

%$$
%\big|V^k_0 - \sup_{\tau\in \mathcal{T}_0(\mathbb{F})}\mathbb{E}[Z(\tau)]\big|^2\le C\Big( \epsilon^2_k\lceil T\epsilon^{-2}_k\rceil^{1-\beta} + \|T - T^k_{e(k,T)}\|_{L^\zeta} + \big(\mathbb{P}\{\delta+T^k_{e(k,T)}< T\}\big)^{\frac{1}{\zeta}} + \delta \ln\big(\frac{2T}{\delta}\big)\Big)
%$$
\noindent for every $k\ge 1$ and $\delta \in (0,1)$.
\end{proposition}
\begin{proof}
Let us fix $0 < \beta < 1$ and $\zeta >1$. A direct application of Corollary 6.1 in \cite{LEAO_OHASHI2017.2} jointly with Assumptions I-II and Lemma 2.2 in \cite{LEAO_OHASHI2017.1} yields
\begin{equation}\label{tail0}
\|Z^k  - Z\|^2_{\mathbf{B}^2}\le \|F\|^2C(\alpha,\sigma,\beta)\epsilon^{2\beta}_k; k\ge 1,
\end{equation}
for a constant $C(\alpha,\sigma,\beta)$. In the sequel, $C$ is a constant which may differ from line to line. Triangle inequality yields

\begin{equation}\label{tail1}
\mathbb{E}\|Z^k(\cdot \wedge T^k_{e(k,1)}) -  Z\|^2_\infty\le 2\mathbb{E}\|Z^k(\cdot \wedge T^k_{e(k,1)}) - Z^{k}\|^2_\infty + 2\mathbb{E}\|Z^k - Z\|^2_\infty=: I^k_1 + I^k_2.
\end{equation}
By (\ref{abcov}) in Theorem \ref{optSTres}, we only need to estimate $I^k_1$. We observe
\begin{equation}\label{tail2}
I^k_1\le C\|F\|^2\mathbb{E}\Big[\max_{T^k_{e(k,1)} < T^k_p \le T^k_{N^k(1)}}|X^{k}(T^k_{e(k,1)}) -X^{k}(T^k_p)|^2\mathds{1}_{\{T^k_{e(k,1)} < 1\}}\Big] + C\mathbb{E}|1- T^k_{e(k,1)}|;k\ge 1.
\end{equation}
We set $E_{k}= \{T^k_{e(k,1)} <  T^k_{N^k(1)}\}$, where we recall $N^k(1)$ is given by (\ref{Nk}). By using Lemma 6.2 in \cite{LEAO_OHASHI2017.2}, Assumption I, Jensen and  H\"{o}lder's inequality for $\zeta >2$, we get

\begin{equation}\label{tail3}
\mathbb{E}\max_{p\ge 1}\Bigg|\int_{T^k_{e(k,T)}}^{T^k_p}\alpha\big(\bar{s}_k,X^{k}\big)ds\Bigg|^2\mathds{1}_{E_{k}\cap G^k_p}\le C\|1 - T^k_{e(k,1)}\|_{L^\zeta};k\ge 1,
\end{equation}
for a constant $C$ which depends on $\alpha$. Here, $G^k_p = \{T^k_{e(k,1)} < T^k_p\le T^k_{N^k(1)}\}$ and we recall $\bar{s}_k$ is given by (\ref{Nk}). For each $\delta>0$ and $k\ge 1$, let us denote

$$E^1_{k,\delta} := \{T^k_{e(k,1)} <  T^k_{N^k(1)}, T^k_{N^k(1)}-T^k_{e(k,1)} > \delta\},~E^2_{k,\delta} :=\{T^k_{e(k,1)} < T^k_{N^k(1)}, T^k_{N^k(1)}- T^k_{e(k,1)} \le \delta\}.$$

For a given $\delta>0$, we split
\begin{small}
\begin{eqnarray*}
\max_{p\ge 1}\Bigg|\int_{T^k_{e(k,1)}}^{T^k_p} \sigma(\bar{s}_k,X^{k})dB(s) \Bigg|^2\mathds{1}_{E_{k}\cap G^k_p}&=&
\max_{p\ge 1}\Bigg|\int_{T^k_{e(k,1)}}^{T^k_p} \sigma(\bar{s}_k,X^{k}) dB(s) \Bigg|^2\mathds{1}_{E^1_{k,\delta}\cap G^k_p}\\
& &\\
&+& \max_{p\ge 1}\Bigg|\int_{T^k_{e(k,1)}}^{T^k_p} \sigma(\bar{s}_k,X^{k})dB(s) \Bigg|^2\mathds{1}_{E^2_{k,\delta}\cap G^k_p}\\
& &\\
&=:& J^k_1(\delta) + J^k_2(\delta).
\end{eqnarray*}
\end{small}
By the additivity of the stochastic integral, we shall apply Burkholder-Davis-Gundy and  H\"{o}lder inequalities jointly with Lemma 6.2 in \cite{LEAO_OHASHI2017.2} to get a constant $C$ which depends on $\sigma$ such that

\begin{equation}\label{tail4}
\mathbb{E}J^k_1(\delta)\le C \big(\mathbb{P}\{\delta+T^k_{e(k,1)}< T^k_{N^k_1}\}\big)^{\frac{1}{\zeta}}\le  C \big(\mathbb{P}\{\delta+T^k_{e(k,1)}< 1\}\big)^{\frac{1}{\zeta}}
\end{equation}
for any $\zeta >2$, $k\ge 1$ and $\delta\in (0,1)$. In order to estimate $J^k_{2}(\delta)$, we need to work with the modulus of continuity of the stochastic integral

$$m_{k}(h) := \sup_{t,s\in [0,1], |t-s|\le h}\Big| \int_{s}^{t} \sigma(s,X^{k})dB(s) \Big|$$
for $h>0$. We notice $J^k_2(\delta)\le m^2_{k}(\delta)~a.s$ for every $k\ge 1$ and $\delta>0$. By applying Th 1 in \cite{fischer} jointly with Assumption I and Lemma 6.2 in \cite{LEAO_OHASHI2017.2}, we arrive at the following estimate

\begin{equation}\label{tail5}
\mathbb{E}J^k_{2}(\delta)\le C \delta\ln \Big(\frac{2}{\delta}\Big)
\end{equation}
for every $k\ge 1$ and $\delta >0$, where $C$ is a constant which only depends on $(\alpha,\sigma)$. Finally, since $e(k,1)=\epsilon^{-2}_k$ and $ T^k_{\epsilon^{-2}_k}=\sum_{i=1}^{\epsilon^{-2}_k}\Delta T^k_n$ is an iid sum of random variables with mean $\epsilon^2_k$, we shall use classical Large Deviations techniques to find for each $q\ge 1$, a constant $C$ (only depending on $q$) such that

\begin{equation}\label{tail6}
\mathbb{P}\{\delta + T^k_{\epsilon^{-2}_k} < 1\}\le \exp\big(-\epsilon^{-2}_k I^* (1-\delta) \big),\quad \mathbb{E}|T^k_{\epsilon^{-2}_k}-1|^q \le C \epsilon^{2q}_k
\end{equation}
for every $k\ge 1$ and $\delta\in (0,1)$. Now, we just need to use (\ref{tail6}) into (\ref{tail2}), (\ref{tail3}) and (\ref{tail4}). By observing the estimates (\ref{tail0}) and (\ref{tail1}), we then conclude the proof.

%$\delta  = \sqrt{\epsilon_k}$, we have

%$$
%\big|V^k_0 - \sup_{\tau\in \mathcal{T}_0(\mathbb{F})}\mathbb{E}[Z(\tau)]\big|^2\le C\Big( \epsilon^{2\beta}_k + \exp\big[-\zeta^{-1}\epsilon^{-2}_k I^* \big(1-\delta \big) \big] +  \delta \ln\big(\frac{2}{\delta}\big)\Big); k\ge 1, \delta\in (0,1).
%$$
%This concludes the proof.

%By noticing that

%$$-\epsilon^{-2}_k I^* \big(1-\sqrt{\epsilon_k}\big) = \inf_{\lambda <0}\Big[ -\lambda\epsilon^{-2}_k(1-\sqrt{\epsilon_k}) +\epsilon^{-2}_k \ln \Big(\frac{1}{\text{cosh}(\sqrt{2|\lambda|})}\Big)  \Big]\downarrow -\infty$$
%as $k\rightarrow +\infty$, we conclude the proof.

%$|T^k_{N^k_T} - T^k_{e(k,T)}|\le |T-T^k_{e(k,T)}|$ a.s,
\end{proof}
\subsection{Non-Markovian SDE driven by Fractional Brownian motion}\label{fbmsection}

In this section, we analyze Theorem \ref{optSTres} when the reward process $Z=F(X)$ is driven by the following SDE

\begin{equation}\label{limsdefbm}
X(t) = x+ \int_0^t\alpha(s,X)ds + B_H(t)
\end{equation}
where $\alpha$ is a bounded non-anticipative functional satisfying Assumption I and $B_H$ is the one-dimensional fractional Brownian motion (henceforth abbreviated by FBM) over the interval $[0,T]$ for $\frac{1}{2}< H < 1$. By Th 3.4 in \cite{Hu}, we shall represent

$$
B_H(t) = \int_0^t \rho_H(t,s)B(s)ds; 0\le t\le T,
$$
where $\rho_H(t,s):=\partial_s K_H(t,s)$ and $K_H(t,s):=d_Hs^{\frac{1}{2}-H}\int_s^t u^{H-\frac{1}{2}}(u-s)^{H-\frac{3}{2}}du; 0< s < t\le T$ is the classical square-integrable Volterra kernel which represents FBM. In order to construct an imbedded discrete structure for $Z$, we need a structure $B^k_H$ for $B_H$ which converges in $\mathbf{B}^1$ and the natural candidate is

$$
B^k_{H}(t): = \int_0^{\bar{t}_k}\rho_H(\bar{t}_k,s)A^k(s)ds;0\le t\le T,
$$
where we recall $\bar{t}_k$ is given by (\ref{Nk}). Due to singularity of the kernel $\rho_H$ jointly with appearance of random times and the piecewise constant process $A^k$, it is not obvious $B^k_H$ fulfills this requirement. The proof of the next result is postponed to the Appendix.
\begin{proposition}\label{FBMapproximation}
If $\frac{1}{2} < H < 1$ and $H-\frac{1}{2} < \lambda < \frac{1}{2}$, there exists a constant $C$ which depends on $H$, $T$ and $\lambda$ such that

$$\mathbb{E}\sup_{0\le t\le T}|B^k_H(\bar{t}_k) - B_H(t)|\le C\epsilon_k^{1-2\lambda}$$
for every $k\ge 1$.
\end{proposition}

The natural candidate for an imbedded discrete structure $\big((Z^k)_{k\ge1},\mathscr{D}\big)$ w.r.t (\ref{rewardSDEBM}) is given by

\begin{equation}\label{ZkFBM}
Z^k(t) := \sum_{n=0}^{\infty}F\big(T^k_n, X^k\big)\mathds{1}_{\{T^k_n\le t < T^k_{n+1}\}}; 0\le t\le T
\end{equation}
where $X^k(0):=x$ and we define recursively

$$
X^{k}(T^k_{n}):=X^{k}(T^k_{n-1}) + \alpha\big(T^k_{n-1},X^{k}\big)\Delta T^k_{n} +\Delta B^k_{H}(T^k_{n})
$$
for $1\le n\le e(k,T)$. We are now able to state the main result of this section. To simplify the presentation, we set $T=1$.

\begin{proposition}\label{exFBM}
If Assumptions \textbf{I-II} hold for $\theta=1$, $X$ satisfies (\ref{limsdefbm}) with $\frac{1}{2} < H < 1$ and $\alpha$ is bounded, then $\big((Z^k)_{k\ge 1},\mathscr{D}\big)$ given by (\ref{ZkFBM}) is an imbedded discrete structure for $Z$. More importantly, for every $H-\frac{1}{2} < \lambda < \frac{1}{2}$, there exists a constant $C$ which depends on $\alpha, H,\lambda$ such that

\begin{equation}\label{rateFBMex}
\big|V^k_0 - \sup_{\tau\in \mathcal{T}_0(\mathbb{F})}\mathbb{E}[Z(\tau)]\big| \le C \epsilon_k^{1-2\lambda}
\end{equation}
for every $k\ge 1$.
\end{proposition}
\begin{proof}
Let us fix $H-\frac{1}{2} < \lambda < \frac{1}{2}$. By repeating the same argument presented in Proposition 6.2 in \cite{LEAO_OHASHI2017.2} jointly with the Lipschitz property of $F$, one can easily check there exists a constant $C(\alpha,\beta,H,\lambda)$ such that
\begin{equation}
\|Z^k - Z\|_{\mathbf{B}^1}\le \|F\|C(\alpha,\beta,H,\lambda)\epsilon^{1-2\lambda}_k
\end{equation}
for every $k\ge 1$ and $\beta\in (0,1)$. By arguing just like the proof of Proposition \ref{exBM}, we have

$$
\big|V^k_0 - \sup_{\tau\in \mathcal{T}_0(\mathbb{F})}\mathbb{E}[Z(\tau)]\big| \le \|F\|C(\alpha,\beta,H,\lambda) \epsilon^{1-2\lambda}_k + \|F\| \sup_{0\le t\le 1}|X(t\wedge T^k_{e(k,1)}) - X(t)|
$$

%\mathbb{E}|1-T^k_{e(k,1)}|\\
%& &\\
%&+&\|F\|\mathbb{E}\Big[\max_{T^k_{e(k,1)} < T^k_p \le T^k_{N^k(1)}}|X^{k}(T^k_{e(k,1)}) -X^{k}(T^k_p)|\mathds{1}_{\{T^k_{e(k,1)} < 1\}}\Big] + =:R^k_1 + R^k_2,
%\end{eqnarray*}
for every $k\ge 1$. We observe

$$ \sup_{0\le t\le 1}|X(t\wedge T^k_{e(k,1)}) - X(t)|\le C\big(1+ \|X^\eta\|_\infty\big) |1-T^k_{e(k,1)}| + \sigma \|B_H(\cdot \wedge T^k_{e(k,1)}) - B_H\|_\infty$$
where $\mathbb{E}\|X\|^{2p}_{\infty}\le C(1+|x_0|^{2p}\exp(C)$ with $p\ge 1$, for a constant $C$ depending on $B_H$. For every $\epsilon>0$, it is well-known there exists $G_{\epsilon}\in \cap_{q\ge 1}L^q(\mathbb{P})$ and a deterministic constant $C$ such that

$$|B_H(t) - B_H(s)|\le C |t-s|^{(H-\epsilon)}G_{\epsilon}~a.s.$$
for every $t,s\in [0,1]$. Therefore,
$$\sup_{0\le t\le 1}|B_H(t) - B_H(t\wedge T^k_{e(k,1)})|\le C |1 - T^k_{e(k,1)}|^{(H-\epsilon)}G_{\epsilon}~a.s.$$
%% ALBERTO. Why is it necessary to introduce a constant $C$?
By using (\ref{tail6}) and Cauchy-Schwarz's inequality, for every $0 < \gamma < H$, there exists a constant $C$ such that

$$\mathbb{E}\|B_H(\cdot \wedge T^k_{e(k,1)}) - B_H\|_\infty\le C \epsilon_k^{2\gamma}.$$
Since, we shall take $\gamma + \lambda > \frac{1}{2}$, we then conclude the proof.
\end{proof}

\subsection{How to compute optimal values ?}\label{appliedsec}
At this point, it is instructive to end this section with some guidelines on how to concretely produce optimal values in a given (possibly non-Markovian) optimal stopping time problem over a a given interval $[0,T]$. For simplicity of exposition, the dimension of the underlying Brownian motion will be set equal to $d=1$. At first, we recall the key objects of our methodology are given by (\ref{DPST}), (\ref{cvalues}) and (\ref{valuef}), where each random element $\mathcal{A}^k_n$ induces an image probability measure $\rho^k_n:=\mathbb{P}^k_n$ on $\mathbb{S}^n_k; n\ge 1$ where $\rho_0$ is just the Dirac concentrated on $\textbf{0}$.

Let $\mathbb{Z}^k=\{\mathbb{Z}^k_n; n=0, \ldots, e(k,T)\}$ be a list of Borel functions which realizes (\ref{Zbb}). Recall $\mathbb{Z}^k$ must be interpreted as the payoff functional composed with a pathwise version of an imbedded discrete structure (see Definition \ref{GASdef}) for a given state which is typically an Euler-type approximation driven by $A^k$. See \cite{ohashi, LEAO_OHASHI2017.2} for further details. We assume that $\mathbb{Z}^k_n:\mathbb{S}^n_k\rightarrow \mathbb{R} \in L^2(\mathbb{S}^n_k,\rho^k_n)$ for every $n=0, \ldots, e(k,T)$. Let us now select a subset $\{\widehat{\mathbf{U}}^k_j;j=0, \ldots e(k,T)-1 \}$  of functions such that $\widehat{\mathbf{U}}^k_j\in L^2(\mathbb{S}^{j}_k,\rho^k_j)$ for each $j=0, \ldots, e(k,T)-1$. For each choice of functions, we set inductively

\begin{equation}\label{DPST1}
\left\{\begin{array}{l}
 \widehat{\tau}^{k}_{e(k,T)}:= e(k,T) \\
\widehat{\tau}^{k}_{j}:= j 1\!\!1_{\widehat{G}^k_j} + \tau^{k}_{j+1}1\!\!1_{(\widehat{G}^{k}_j)^c}; 0\le j \le e(k,T)-1,
\end{array}\right.
\end{equation}
where $\widehat{G}^k_j: = \{\mathbb{Z}^k_j (\mathcal{A}^k_j)\ge \widehat{\mathbf{U}}^{k}_j(\mathcal{A}^k_j)\}; 0\le j\le e(k,T)-1$ and $\widehat{\tau}^{k}=\widehat{\tau}^{k}_0$. Here, $\widehat{\mathbf{U}}^k_j(\cdot)$ should be interpreted as a suitable approximation of $\mathbb{E}\big[\mathbb{Z}^k_{\widehat{\tau}^{k}_{j+1}}(\mathcal{A}^k_{\widehat{\tau}^{k}_{j+1}})|\mathcal{A}^k_j=\cdot\big]$ for each $j=0, \ldots, e(k,T)-1$.

The set $\{\widehat{\tau}^{k}_j; 0\le j \le e(k,T)\}$ induces a set of conditional expectations $$\mathbb{E}\Big[Y^k(\widehat{\tau}^{k}_{j+1})|\mathcal{A}^k_{j}\Big]= \mathbb{E}\Big[\mathbb{Z}^k_{\widehat{\tau}^{k}_{j+1}}(\mathcal{A}^k_{\widehat{\tau}^{k}_{j+1}})\big|\mathcal{A}^k_j\Big]; j=0, \ldots, e(k,T)-1,$$
so that one can postulate

$$\max\big\{\mathbb{Z}^k_0(\textbf{0}); \widehat{\mathbf{U}}^{k}_0(\textbf{0})\big\}$$
as a possible approximation for (\ref{valuef}). Having said that, we are now able to produce a Longstaff-Schwartz-type Monte-Carlo scheme as demonstrated by \cite{ohashi} which we briefly outline here. For further details, we refer the reader to \cite{ohashi}.

Given any positive integer $N$, select $\mathcal{H}^k_{N,0}\subset \mathbb{R}$. For each $j=1, \ldots, e(k,T)-1$, select $\mathcal{H}^k_{N,j}\subset L^2(\mathbb{S}^j_k, \rho^k_j)$. The sets $\mathcal{H}^k_{N,j}; 0\le j\le e(k,T)$ (the so-called approximation architectures) possibly depend on $N$ and their choice are dictated by some a priori information that one possibly has about the continuation values (\ref{cvalues}). From $\Big(\mathcal{A}^k_\ell; 0\le \ell \le e(k,T)\Big)$, generate $N$ independent samples
$$\mathcal{A}^k_{0,i}, \mathcal{A}^k_{1,i}, \ldots, \mathcal{A}^k_{e(k,T),i}; i=1, \ldots, N.$$
This step can be implemented by means of the perfect simulation algorithm developed by \cite{Burq_Jones2008}. For a concrete implementation in the context of hedging for European-type options, see \cite{bonetti}. For each $\ell=0, \ldots ,e(k,T)$, let us denote

$$\mathbf{A}^k_{\ell N}:=\big(\mathcal{A}^k_{\ell,1},\mathcal{A}^k_{\ell,2}, \ldots, \mathcal{A}^k_{\ell,N}; \ldots; \mathcal{A}^k_{e(k,T), 1},\mathcal{A}^k_{e(k,T), 2}, \ldots, \mathcal{A}^k_{e(k,T),N} \big)$$
with $(e(k,T)-\ell+1)N$-factors.

For $j=e(k,T)-1$, we set $\widehat{\tau}^k_{j+1} := \widehat{\tau}^k_{j+1}(\mathbf{A}^k_{(j+1)N}) := e(k,T)$ and generate $\{  (\mathcal{A}^k_{j,i}, (\mathbb{Z}^k_{\widehat{\tau}^k_{j+1}})_i); 1\le i\le N\}$, where we define $(\mathbb{Z}^k_{\widehat{\tau}^k_{j+1}})_i:=\mathbb{Z}^k_{e(k,T)}(\mathcal{A}^k_{e(k,T),i}); 1\le i\le N.$ We then select

\begin{equation}\label{mini}
\widehat{\mathbf{U}}^k_{j} := \text{arg min}_{g\in \mathcal{H}^k_{N,j}}\frac{1}{N}\sum_{i=1}^N \Big((\mathbb{Z}^k_{\widehat{\tau}^k_{j+1}})_i - g(\mathcal{A}^k_{j,i})     \Big)^2.
\end{equation}
One should notice that $\widehat{\mathbf{U}}^k_{j}$ is a functional of $\mathbf{A}^k_{j N}$ and we assume the existence of a minimizer (see Remark 4.4 in \cite{ohashi})

\begin{equation}\label{minimizerMC}\widehat{\mathbf{U}}^k_{j}:\mathbb{S}^{j}_k\times \big(\mathbb{S}^{j}_k\big)^N \times \ldots \times \big(\mathbb{S}^{e(k,T)}_k\big)^N\rightarrow \mathbb{R}
\end{equation}
of (\ref{mini}) which possibly depend on $N$. With $\widehat{\mathbf{U}}^k_j$ at hand, we compute $\widehat{\tau}^k_{ji} = \Big( \widehat{\tau}^k_j \big( \mathbf{A}^k_{jN} \big) \Big)_i$, the value that $\widehat{\tau}^k_{j} =\widehat{\tau}^k_{j}\big( \mathbf{A}^k_{jN} \big) $ assumes based on the $i$-th sample according to (\ref{DPST}). In this case, we set

\begin{equation}\label{DPST3}
\Big(\mathbb{Z}^k_{\widehat{\tau}^k_j}\Big)_i:=\left\{
\begin{array}{rl}
\mathbb{Z}^k_j\big(\mathcal{A}^k_{j,i}\big); & \hbox{if} \ \widehat{\tau}^k_{ji}=j \\
\mathbb{Z}^k_{\widehat{\tau}^k_{(j+1)i}}\big(\mathcal{A}^k_{\widehat{\tau}^k_{(j+1)i},i}\big);& \hbox{if} \ \widehat{\tau}^k_{ji} = \widehat{\tau}^k_{(j+1)i}\\
\end{array}
\right.
\end{equation}
where $\widehat{\tau}^k_{(j+1)i}=e(k,T)$ for $1\le i\le N$.

Based on (\ref{mini}) and (\ref{DPST3}), we then repeat this procedure inductively $j=e(k,T)-2, \ldots, 1, 0$ until step $j=0$ to get

$$\Big( \widehat{\tau}^k_{ji}, \widehat{\mathbf{U}}^k_{j}, \big(\mathbb{Z}^k_{\widehat{\tau}^k_{j}}\big)_i   \Big); 0\le j\le e(k,T), 1\le i\le N.$$

For $j=0$, we set

$$\widehat{V}_0(\mathbf{A}^k_{0N}) := \text{max}\Big\{\mathbb{Z}^k_0(\textbf{0}), \widehat{\mathbf{U}}^k_0(\mathbf{A}^k_{0N})\Big\}.$$
The above procedure is consistent with the limiting optimal stopping time problem as demonstrated by the following result. It is an immediate consequence of Theorem 4.1 in \cite{ohashi} and Propositions \ref{exBM} and \ref{exFBM}. In the sequel, $vc$ denotes the Vapnik-Chervonenkis dimension of a subset.

\

\noindent \textbf{(H1)} For each $k\ge 1$, suppose that $\mathcal{H}^k_{N,j}\subset L^2(\mathbb{S}^j_k, \rho^k_j)$ and there exists $\nu_k$ such that $vc\big(\mathcal{H}^k_{N,j}\big)\le \nu_k < +\infty$ for every $j=1, \ldots, e(k,T)-1$ and for every $N\ge 2$.

\

\noindent \textbf{(H2)} For each $k\ge 1$, there exists $B_k$ such that $\sup\{\|f\|_\infty; f\in \mathcal{H}^k_{N,j}\}\le B_k < +\infty$ for every $j=0, \ldots, e(k,T)-1$ and $N\ge 2$.

\

By applying Propositions \ref{exBM} and \ref{exFBM} and Th 4.1 in \cite{ohashi}, we arrive at the following result.
\begin{corollary}\label{mainTHLS}
Assume \textbf{(A1-A2-H1-H2)} and the hypotheses of Propositions \ref{exBM} and \ref{exFBM} hold true. Let us assume the architecture spaces $\mathcal{H}^k_{N,j}$ are dense subsets of $L^2(\rho^k_j)$ for each $j=1, \ldots, e(k,T)-1$, for each positive integer $N\ge 2$ and $k\ge 1$. Then, for every $k\ge 1$ sufficiently large

\begin{equation}\label{a.sconv2}
\lim_{N\rightarrow +\infty}|\widehat{V}^k_0(\mathbf{A}^k_{0N}) - S(0)|= 0~a.s.
\end{equation}
where $S(0) = \sup_{\tau\in \mathcal{T}_0(\mathbb{F})}\mathbb{E}[Z(\tau)]$ and $Z$ is given by (\ref{rewardSDEBM}).
\end{corollary}

In order to compute a precise number of steps in our scheme, we just need to apply Propositions \ref{exBM} and \ref{exFBM} combined with Corollary 4.1 in \cite{ohashi}. In this case, similar to the classical Markovian case, smoothness of the continuation values $\mathbf{U}^{k}_j; j=0,\ldots, e(k,T)-1$ is crucial for how to choose approximation spaces to get
the most favorable rate of convergence by properly balancing the approximation error and the sample error.

\begin{example}\label{EXAMPLE_PAPER}
For simplicity of exposition, let $X$ be the state process given in Section \ref{fbmsection}, where the terminal time, the level of discretization $\epsilon_k$, the drift and the payoff are defined, respectively, by

$$T=1, \quad \epsilon_k=\varphi(k),\quad \alpha (s,\eta) = b(\eta(s)),\quad F:\Lambda\rightarrow \mathbb{R}; \eta \mapsto F(\eta):=h(\eta(1))$$
for each $\eta \in D([0,1];\mathbb{R})$ and $s\in [0,1]$, where $b$ and $h$ are Lipschitz bounded functions and $\varphi:[0,\infty)\rightarrow [0,\infty)$ is a strictly decreasing function (with inverse $\xi$) such that $\sum_{k\ge 1} \varphi^2(k)< \infty$. We fix $\frac{1}{2}< H < 1$ and $\lambda$ as described in Proposition \ref{exFBM}.

Next, we investigate the global numerical error $\textbf{e} = \text{e}_1+ \text{e}_2$ one may occur in a concrete fully non-Markovian example. The error $\textbf{e}$ can be decomposed as the sum of two terms: the first one ($\text{e}_1$), which we study in this paper, is the discrete-type filtration approximation error; the second one ($\text{e}_2$), which we study in \cite{ohashi}, it is related to the numerical approximation of the conditional expectations associated with the continuation values. By using the fact that $\alpha$ and $F$ only depend on the present and performing a similar computation as described in the proof of Th 5.1 in \cite{ohashi}, $\mathbf{b}^k_n \mapsto \mathbf{U}^k_n(\mathbf{b}^k_n)$ is Lipschitz from $\widetilde{\mathbb{S}}^n_k$ to $\mathbb{R}$, where $\widetilde{\mathbb{S}}^n_k:=((0,+\infty)\times B_r(0)) \times \ldots \times ((0,+\infty)\times B_r(0))$ (n-fold cartesian product) and $B_r(0)$ is an open set centered at the origin with radius $r>1$. We apply Corollary 4.1 in \cite{ohashi} and Proposition \ref{exFBM} to state that

\begin{equation}\label{dis1}
\mathbb{E}|\widehat{V}_0(\mathbf{A}^k_{0N}) - V^k_0|= O\big(\text{log}(N)N^{\frac{-2}{2+e(k,1)-1}}\big),
\end{equation}
\begin{equation}\label{dis2}
|V^k_0 - S(0)| = O (\epsilon_k^{1-2\lambda} ).
\end{equation}
With the estimates (\ref{dis1}) and (\ref{dis2}) at hand, we are now able to infer the amount of work (complexity) to recover the optimal value for a given level of accuracy $\textbf{e}$. Indeed, let us fix $0 < \text{e}_1 < 1$. Equation (\ref{dis2}) allows us to find the necessary number of steps related to the discretization as follows. We observe $\epsilon^{1-2\lambda}_k \le e_1 \Longleftrightarrow k \ge \xi(\text{e}^{\frac{1}{1-2\lambda}}_1)$ and with this information at hand, we shall take $k^* = \xi(\text{e}^{\frac{1}{1-2\lambda}}_1)$. This produces

$$e(k^*,1) = \Big\lceil \frac{1}{\varphi^{2}(k^*)}\Big\rceil$$
number of steps associated with the discretization procedure. For instance, if $\varphi(k) = 2^{-k}$, $\text{e}_1 = 0.40$, $H=0.6$, then we shall take $\lambda=0.15$ and $k^*=-\frac{\text{log}_2~0.40}{0.70}=1.88$. This produces $e(k^*,1)= \lceil 2^{2\times 1.88}\rceil=14$ number of steps. Of course, as $\text{e}_1 \downarrow 0$, the number of steps $e(k^*,1)\uparrow +\infty$, e.g., if $\text{e}_1 = 0.2$, then $k^*=-\frac{\text{log}_2~0.2}{0.70}=3.31$, $e(k^*,1) = 99$ and so on. For a given prescribed error $0 < \text{e}_2 < 1$ and $k^{*}$, equation (\ref{dis1}) allows us to find the necessary number $N$ for the Monte-Carlo scheme in such way that $\mathbb{E}|\widehat{V}_0(\mathbf{A}^{k^*}_{0N}) - V^{k^*}_0|=O(\text{e}_2)$.
\end{example}

\section{Appendix: Proof of Proposition \ref{FBMapproximation}}
In the sequel, $C$ is a constant which may differ from line to line. Let us denote
$$\rho_{H,1}(t,s) =t^{(H-\frac{1}{2})}s^{(\frac{1}{2}-H)}(t-s)^{(H-\frac{3}{2})},\quad \rho_{H,2}(t,s) =s^{(-H-\frac{1}{2})}\int_s^tu^{(H-\frac{3}{2})}(u-s)^{(H-\frac{1}{2})}du$$
for $ 0 < s  <t$. Clearly, there exists a constant $C$ which depends on $T$ such that

\begin{equation}\label{estRHOH}
|\rho_H(t,s)|\le C \{\rho_{H,1}(t,s) + \rho_{H,2}(t,s)\}
\end{equation}
for $ 0 < s  <t$. Triangle inequality and the  H\"{o}lder property of FBM yield

\begin{equation}\label{in}
\sup_{0\le t\le T}|B^k_H(\bar{t}_k) - B_H(t)|\le\sup_{0\le t\le T}|B^k_H(\bar{t}_k) - B_H(\bar{t}_k)| + \|B_H\|_{H-\eta} \Big(\bigvee_{n=1}^{N^k(T)} \Delta T^k_n\Big)^{H-\eta}~a.s
\end{equation}
for every $0 < \eta < H$, where $\|\cdot\|_{\theta}$ denotes the  H\"{o}lder norm for $0 < \theta \le 1$. Clearly

\begin{equation}\label{in2}
\sup_{0\le t\le T}\int_0^{\bar{t}_k} \rho_{H,1}(\bar{t}_k,s)|A^k(s) - B(s)|ds\le \epsilon_k T^{H-\frac{1}{2}}~a.s.
\end{equation}
The delicate component is $\int_0^{\bar{t}_k} \rho_{H,2}(\bar{t}_k,s)|A^k(s) - B(s)|ds$. Let us fix $p>1$ such that

\begin{equation}\label{cond1}
\frac{1}{2} + \frac{1}{p} > \frac{1}{2} > \lambda > \frac{1}{p} + H -\frac{1}{2}.
\end{equation}
so that we must have $p > \frac{1}{1-H}> 2$. Let $q>1$ be a conjugate exponent such that $(-\frac{1}{2}-H+\lambda)q + 1 >0$. The choice (\ref{cond1}) implies

$$\mathbb{E}\Bigg\{\int_0^u s^{-\lambda p}|A^k(s)-B(s)|^pds\Bigg\}^{\frac{1}{p}}< \infty$$
for every $0\le u\le T.$ We observe

\begin{eqnarray*}\label{inn1}
\sup_{0\le t\le T}\int_0^{\bar{t}_k}\rho_{H,2}(\bar{t}_k,s)|A^k(s)-B(s)|ds&\le&C \int_0^{T}u^{\lambda -1+ \frac{1}{q}}(d_k(u))^{\frac{1}{p}}u^{H-\frac{3}{2}}\mathds{1}_{\{N^k(u)>0\}}du\\
& &\\
&+& C \int_0^{T}u^{\lambda-1+\frac{1}{q}}(d_k(u))^{\frac{1}{p}}u^{H-\frac{3}{2}} \mathds{1}_{\{N^k(u)=0\}} du
\end{eqnarray*}
where $C$ only depends on $H$ and
$$
d_k(u):=\int_0^u s^{-\lambda p}|A^k(s)-B(s)|^pds.
$$
To shorten  notation, let us denote
$$J^k(u) = \mathbb{E}\Bigg[\int_0^{T^k_{N^k(u)+1}} s^{-\lambda p}|A^k(s)-B(s)|^pds\Bigg]\mathds{1}_{\{N^k(u)>0\}}.$$
We observe

$$d_k(u)\mathds{1}_{\{N^k(u)=0\}}\le \|B\|^p_{\lambda}u\mathds{1}_{\{N^k(u)=0\}},\quad d_k(u)\mathds{1}_{\{N^k(u)>0\}}\le J^k(u)$$
a.s for every $u\in [0,T]$. Let us now estimate $J^k(u)$ for a given $u\in (0,T]$. A lengthy but straightforward calculation based on the strong Markov property of the shifted Brownian motion over the stopping times $T^k_n$ yields

\begin{eqnarray*}
\mathbb{E}[J^k(u)] &\le& \mathbb{E}\Bigg[\sum_{n=1}^{2N^k(u)}\int_{0}^{\Delta T^k_n}|v+T^k_{n-1}|^{-\lambda p}|B(v+T^k_{n-1})-B(T^k_{n-1})|^pdv\Bigg]\\
& &\\
&=& C \mathbb{E}\Big[[A^k](u)\Big]\epsilon_k^{2(-\lambda p+ \frac{p}{2})}\\
& &\\
&\le& C\mathbb{E}\Big[\sup_{0\le \ell \le u}|B(\ell)|^2\Big]\epsilon_k^{2(-\lambda p+ \frac{p}{2})}\le Cu\epsilon_k^{2(-\lambda p+ \frac{p}{2})}
\end{eqnarray*}
for a constant $C$ which only depends on $H$. Then,

$$
\mathbb{E}(d_k(u))^{\frac{1}{p}}\le C u^{\frac{1}{p}}\epsilon_k^{2(-\lambda+\frac{1}{2})} + \Big(\mathbb{E}\|B\|^{\gamma p}_\lambda\Big)^{\frac{1}{p\gamma}} \mathbb{P}^{\frac{1}{\beta p}}\{N^k(u)=0\} u^{\frac{1}{p}}
$$
for a constant $C$ which depends on $(\lambda,p)$, where $\gamma,\beta>1$ are conjugate exponents. Moreover,

$$\mathbb{P}^{\frac{1}{p\beta}}\{N^k(u)=0\}\le C u^{-\frac{1}{p\beta}}\epsilon_k^{\frac{2}{p\beta}}; u >0, p\beta >2,$$
for a constant $C$ which only depends on $p\beta$. This implies

\begin{equation}\label{in3}
\mathbb{E}\sup_{0\le t\le T}\int_0^{\bar{t}_k}\rho_{H,2}(\bar{t}_k,s)|A^k(s)-B(s)|ds\le C T^{H-\frac{1}{2}+\lambda}\epsilon_k^{2(-\lambda+\frac{1}{2})} + C\epsilon_k^{\frac{2}{p\beta}}T^{H-\frac{1}{2}+\lambda-\frac{1}{p\beta}}
\end{equation}
by adjusting $\frac{1}{p\beta} < \lambda < \frac{1}{2}$. Summing up the steps (\ref{estRHOH}), (\ref{in}), (\ref{in2}), (\ref{in3}) and using Fernique's theorem, Lemma 2.2 in \cite{LEAO_OHASHI2017.1} and H\"{o}lder's inequality, there exists a constant $C = C(\theta,H,a,b,\lambda,\beta,\eta)$ such that

$$
\mathbb{E}\sup_{0\le t\le T}|B^k_H(\bar{t}_k) - B_H(t)|\le C(\theta,H,a,b,\lambda,\beta,\eta)\Big(\epsilon_k^{2(H-\eta)-\frac{2}{b}(1-\theta)} + \epsilon^{1-2\lambda}_k + \epsilon^{\frac{2}{p\beta}}_k\Big)
$$
for every $a,b >1$ conjugate exponents, $\theta\in (0,1)$, $(\lambda,p)$ satisfying (\ref{cond1}), $\beta>1$ and $\eta \in (0,H)$. Finally, we can choose $(\eta,\theta,b)$ in such way that

$$1 > 2\lambda>2\eta +\frac{2}{b}(1-\theta)$$
and since $2H > 1$, we then have $1-2\lambda < 2H - \big( 2\eta +\frac{2}{b}(1-\theta)\big)$. We also can choose $(p,\beta)$ in such way that

$$\frac{1}{2}> \lambda > \frac{1}{2}-\frac{1}{p\beta}> \frac{1}{p} + H-\frac{1}{2}$$
and in this case $1-2\lambda < \frac{2}{p\beta}$. This concludes the proof.

\newpage 
\textbf{Acknowledgement:} Alberto Ohashi and Francesco Russo acknowledge the financial support of Math Amsud grant 88887.197425/2018-00. Alberto Ohashi acknowledges ENSTA ParisTech for the hospitality and CNPq-Bolsa de Produtividade de Pesquisa grant 303443/2018-9. The authors would like to thank the Referees for the careful reading and suggestions which considerably improved the presentation of this paper.

%\acks Alberto Ohashi and Francesco Russo acknowledge the financial support of Math Amsud grant 88887.197425/2018-00. Alberto Ohashi acknowledges ENSTA ParisTech for the hospitality and CNPq-Bolsa de Produtividade de Pesquisa grant 303443/2018-9. The authors would like to thank the Referees for the careful reading and suggestions which considerably improved the presentation of this paper.


\begin{thebibliography}{99}
\bibitem{agarwal} {\sc Agarwal, A., Juneja, S. and Sircar, R.} (2016). American options under stochastic volatility:
control variates, maturity randomization and multiscale asymptotics. \textit{Quant. Finance}, \textbf{16}, 1, 17-30.

\bibitem{bally} {\sc Bally, V. and Pag\`es, G.} (2003). A quantization algorithm for solving discrete time multidimensional optimal stopping problems,
\textit{Bernoulli}, \textbf{9}, 6, 1003-1049.

\bibitem{bally1} {\sc Bally V. and G, Pag\`es} (2003). Error analysis of the quantization algorithm
for obstacle problems. \textit{Stochastic Process. Appl.}, \textbf{106}, 1-40.


\bibitem{bayer} {\sc Bayer, C., Friz, P. and Gatheral, J.} (2016).
 Pricing under rough volatility. \textit{Quant. Finance}, \textbf{16}, 6, 887--904.

\bibitem{belomestny} {\sc Belomestny, D.} (2013). Solving optimal stopping problems via dual empirical dual optimization. \textit{Ann. Appl. Probab}, \textbf{23}, 5, 1988--2019.

\bibitem{bertsekas} {\sc Bertsekas, D.P. and Shreve, S.} Stochastic optimal control: The discrete-time case. Athena Scientific Belmont Massachusetts, 1996.

\bibitem{ohashi} {\sc Bezerra, S.C., Ohashi, A. and Russo, F.} (2017). Discrete-type approximations for non-Markovian optimal stopping problems: Part II. arXiv:1707.05250.

\bibitem{bonetti} {\sc Bonetti, D., Le\~ao, D., Ohashi, A. and Siqueira, V.} (2015). A general multidimensional
Monte-Carlo approach for dynamic hedging under stochastic volatility, \textit{Int. Journal of Stochastic Analysis}, Article ID 863165.


\bibitem{borodin} {\sc Borodin, A. N. and Salminen, P.} \textit{Handbook of Brownian Motion: Facts and
Formulae}. Birkhauser, 2002.

\bibitem{bouchard} {\sc Bouchard, B. and Warin, X.} Monte-Carlo valuation of American options: facts and new algorithms to improve existing methods. \textit{Numerical methods in finance}. Springer, Berlin, Heidelberg, 2012. 215-255.


\bibitem{bouchard1} {\sc Bouchard B. and Chassagneux, J.F} (2008). Discrete time approximation for
continuously and discretely reflected BSDEs. \textit{Stochastic Process. Appl.}, \textbf{118}, 12, 2269-2293.

\bibitem{bouchard2} {\sc Bouchard B. and Touzi, N.} (2004). Discrete-Time Approximation and Monte-Carlo Simulation of Backward Stochastic Differential Equations. \textit{Stochastic Process. Appl.}, \textbf{111}, 2, 175-206.

\bibitem{Burq_Jones2008} {\sc Burq, Z. A. and Jones, O. D.} (2008).
Simulation of Brownian motion at first-passage times. \textit{Math. Comput. Simul.} \textbf{77}, 1, 64--71.


\bibitem{carmona} {\sc Carmona, R., Del Moral, P., Hu, P. and Oudjane,N.} (Eds.). \textit{Numerical Methods in finance}. Bordeaux, June 2010 Series: Springer Proceedings in Mathematics, Vol. 12, 2012, XVII.

\bibitem{viens} {\sc Chronopoulou, A. and Viens, F.} (2012). Stochastic volatility models with long-memory in discrete and continuous time. \textit{Quant. Finance}, \textbf{12}, 4, 635-649.

\bibitem{dellacherie} {\sc Dellacherie, C. and Meyer, P. A.}
\textit{Probabilities and Potential}.
Probabilities and potential. North-Holland Mathematics Studies, 29. North-Holland Publishing Co., Amsterdam-New York, 1978.


\bibitem{elkaroui} {\sc El Karoui, N., Kapoudjian, C., Pardoux, E., Peng, S.and Quenez, M.C.} (1997). Solutions of backward SDE's, and related obstacle problems for PDE's . \textit{Ann. Probab.}~25, \textbf{2}, 702--737.

\bibitem{ekren} {\sc Ekren, I.} (2017). Viscosity solutions of obstacle problems for fully nonlinear path-dependent PDEs. \textit{Stochastic Process. Appl.} \textbf{127}, 12, 3966-3996.

\bibitem{fischer} {\sc Fischer, M. and Nappo, G.} (2009). On the Moments of the Modulus of Continuity of It\^o Processes. \textit{Stoch. Analyis Appl}, 28, 1, 103-122.

\bibitem{fuhrman} {\sc Fuhrman, M., Pham, H. and and Zeni, F.} (2016). Representation of non-Markovian optimal stopping problems by constrained BSDEs with a single jump. \textit{Electron. Commun. Probab,} \textbf{21}, 3, 7 pp.

\bibitem{forde}  {\sc Forde, M. and Zhang, H.} (2017). Asymptotics for Rough Stochastic Volatility Models
 \textit{SIAM J. Finan. Math}, \textbf{8}, 1, 114--145.


\bibitem{gatheral} {\sc Gatheral, J.} Volatility is rough. Available at SSRN 2509457, 2014.

\bibitem{he} {\sc He, S-w., Wang, J-g., and Yan, J-a.} \textit{Semimartingale Theory and Stochastic Calculus}, CRC Press, 1992.

\bibitem{Hu} {\sc Hu, Y.} (2005). Integral transformations and anticipative calculus for fractional Brownian motions. \textit{Mem. Amer. Math. Soc}, \textbf{175}, 825.

\bibitem{koshnevisan} {\sc Khoshnevisan, D. and Lewis, T.M.} (1999). Stochastic calculus for Brownian motion on a Brownian fracture. \textit{Ann. Appl. Probab}. \textbf{9}, 3, 629--667.

\bibitem{lamberton} {\sc Lamberton, D.} \textit{Optimal stopping and American options}. Daiwa Lecture Ser., Kyoto, 2008.

\bibitem{LEAO_OHASHI2013} {\sc Le\~{a}o, D. and Ohashi, A.} (2013). Weak approximations for Wiener functionals. \textit{Ann. Appl. Probab,} 23, \textbf{4}, 1660--1691.

\bibitem{LEAO_OHASHI2017.1} {\sc Le\~ao, D. Ohashi, A. and Simas, A. B.} (2018). A weak version of path-dependent functional It\^o calculus. \textit{Ann. Probab}, \textbf{46}, 6, 3399-3441.

\bibitem{LEAO_OHASHI2017.2} {\sc Le\~{a}o, D., Ohashi, A. and Souza, F.} (2018).
 Stochastic near-optimal controls for path-dependent systems. arXiv: 1707.04976v2.

\bibitem{LEAO_OHASHI2019} {\sc Le\~{a}o, D., Ohashi, A. and Souza, F.}
 Mean variance hedging with rough stochastic volatility. In preparation.



\bibitem{rambharat} {\sc Rambharat, B. R., and Brockwell, A. E.} (2010). Sequential Monte-Carlo Pricing of American-Style Options
under Stochastic Volatility Models. \textit{Ann. Appl. Stat}, \textbf{4}, 1, 222--265.

\bibitem{rogers} {\sc Rogers, L. C. G.} (2002). Monte-Carlo valuation of American options. \textit{Math. Finance}, \textbf{12}, 271--286.

\bibitem{schoenmakers} {\sc Schoenmakers, J., Zhang, J. and Huang, J.} (2013). Optimal Dual Martingales, Their Analysis, and Application to New Algorithms
for Bermudan Products. \textit{Siam J. F. Math}, \textbf{4}, 86–116.


\bibitem{zanger} {\sc Zanger, D.} (2018). Convergence of a least-squares Monte-Carlo algorithm for American option pricing with depedendent sample data. \textit{Math. Finance}, \textbf{28}, 1, 447-479.






























\end{thebibliography}
\end{document}